\documentclass{article}
\usepackage{amsmath}
\usepackage{amsthm}
\usepackage{latexsym}
\usepackage{verbatim}
\usepackage{amscd}
\usepackage{amssymb}
\usepackage{tikz}
\newtheorem{theorem}{Theorem}[section]
\newtheorem{definition}[theorem]{Definition}
\newtheorem{lemma}[theorem]{Lemma}
\newtheorem{prop}[theorem]{Proposition}
\newtheorem{fact}[theorem]{Fact}

\newtheorem{exa}[theorem]{Example}
\newtheorem{remark}[theorem]{Remark}
\newcommand{\commment}[1]{}
\newcommand{\p}{\mathcal{P}}
\newcommand{\Q}{\mathcal{Q}}
\newcommand{\f}{\mathcal{F}}
\newcommand{\g}{\mathcal{G}}
\newcommand{\s}{\mathcal{S}}
\newcommand{\ii}{\mathcal{I}}
\newcommand{\C}{\mathbb{C}}
\begin{document}

\title{Groupoid Quantales: a non \'etale setting}
\author{Alessandra Palmigiano\footnote{The research of the first author was supported by the VENI grant 639.031.726 of the Netherlands Organization for Scientific Research
    (NWO).}\, and Riccardo Re}
\date{}
\maketitle

\maketitle

\begin{abstract}
\noindent We establish a bijective correspondence involving a class of unital involutive
quantales and a class of non \'etale set groupoids whose space of units is a sober space. This class includes equivalence relations that arise from group actions. The resulting axiomatization of the class of quantales, as well as
the correspondence defined here, extend
the theory of \'etale groupoids
and their quantales \cite{Re07}.

\vspace{2pt} \noindent \textsl{Keywords}: unital involutive
quantale, topological groupoid, non \'etale groupoid.

\vspace{2pt} \noindent  \textsl{2000 Mathematics Subject
Classification}: 06D05, 06D22, 06D50, 06F07, 18B40, 20L05, 22A22,
54D10, 54D30, 54D80.
\end{abstract}
\section{Introduction}
Important examples of groupoids that are non \'etale abound. Typical examples are given by equivalence relations induced from group actions with fixed points. It is then natural to seek  algebraic descriptions of these groupoids, analogously to what is done for instance in \cite{Pat99, Renault, Re07} in the context of groupoids that are \'etale.
In this paper, we  establish a bijective correspondence involving a class of unital involutive
quantales and a class of %not necessarily  \'etale
groupoids whose set of units is a sober topological space. This correspondence extends, in a spatial setting,
the correspondence between localic \'etale groupoids
and inverse  quantal frames defined in \cite{Re07}.

The correspondence in \cite{Re07} has also been extended beyond the \'etale setting in \cite{PrRes09}, to a correspondence  between open groupoids and {\em open quantal frames}. As their name suggests, these quantales satisfy the frame distributivity condition, but are not required to be unital  (the inverse quantal frames in \cite{Re07} are exactly the {\em unital} open quantal frames in \cite{PrRes09}). The correspondence defined in this paper covers an alternative extension: our quantales are  unital, but do not need to be frames.

As already observed in \cite{PrRes09}, the essential difference between the groupoid-quantale correspondence in the \'etale and in  the non-\'etale setting lays in the role played by the inverse semigroup of $G$-sets of a groupoid. Indeed, all the information needed to reconstruct any \'etale groupoid is encoded in the inverse semigroup formed by the germs of its local bisections. The quantales associated with both \'etale groupoids and inverse semigroups, i.e.\ the inverse quantal frames, being characterized as the free join completions of the inverse semigroups, contain no extra information than the inverse semigroups themselves. %In particular, what extra information is encoded in the quantales that are join-generated from that inverse semigroup is not relevant for the correspondence in the {\'etale} setting.
However,  in the non-\'etale setting,  the inverse semigroup of $G$-sets is not enough to reconstruct the groupoid:  the missing information  governs the various possible ways in which any two $G$-sets of the groupoid intersect one another (notice that this is exactly the information content that becomes trivial in \'etale groupoids, because $G$-sets are closed under finite intersection). This extra information is stored in the quantale, which is why quantales are essential to this setting. %In the non-\'etale case, the inverse semigroup is not sufficient anymore to reconstruct the groupoid,
In this paper, %which only deals with  point-set groupoids,
the role of germs in the reconstruction process is played by the classes of an equivalence relation that we refer to as the {\em incidence} relation, which encodes information on the incidence of any two $G$-sets at a point, in the language of quantales.

\commment{
 The observations about this difference has consequences to the theory of $C^\ast$-algebras: indeed, all the standard procedures for associating  $C^\ast$-algebras with topological groupoids $\g$ we are aware of actually associate the $C^\ast$-algebra with the inverse semigroup of local bisections of the groupoid (cf.\ \cite{Pat99}). By the remarks above, this means that the resulting $C^\ast$-algebra is the same as the one canonically associated with the \'etale groupoid $\tilde{\g}$ of germs of the inverse semigroup, which covers $\g$, in the sense that there is a surjective continuous map $\tilde{\g}\to\g$. If one wants to associate a $C^\ast$-algebra with $\g$ depending on its set structure and topology, we would need to associate a $C^\ast$ algebra directly with the quantale $\Q(\g,\mathcal{S})$,  $\mathcal{S}$ being the inverse semigroup of local bisections of $\g$. Indeed it is this quantale, and not the inverse semigroup only, that encodes information about the set and topological structure of $\g$.
}
As mentioned early on, our quantales are not in general frames. Correspondingly, their associated groupoids do not have a topological (or localic) structure on their spaces of arrows $G_1$. In place of topologies  we use designated  collections of $G$-sets that we refer to as {\em selection bases} (cf.\ Definition \ref{def:selfam}). Indeed, rather than being purely between quantales and groupoids, our correspondence is between quantales and {\em pairs} $(\g, \s)$ of groupoids and selection bases. In fact, these pairs can be regarded as categories on the topology of the space of units $G_0$ (cf.\ Remark \ref{rem:towards pointfree gener}). This observation paves the way to a pointfree generalization of this correspondence, which we develop in the companion paper \cite{PaRe10a}.

\begin{comment} As a consequence of the observations above about the $G$-sets and the selection bases, the correspondence associating a $C^\ast$-algebra with a non-\'etale groupoid $\g$ cannot be mediated by the inverse semigroup $\s$ of local bisections of $\g$, but rather by the quantale $\Q(\g,\mathcal{S})$, which encodes information about the topological structure of $\g$.\end{comment}

The results we present in this paper find their main motivation in a much wider research program that seeks noncommutative extensions of the Gelfand-Naimark duality \cite{Co94}. Interestingly, these results have also many points in common with and are potentially relevant to another area of research (which as far as our knowledge goes is disconnected from the first). This area belongs to order theory, algebra and logic, seeks  representability results for classes of relation algebras, and its research program is well exemplified by \cite{JoTa52}, where a certain class of relation algebras is concretely represented via groupoids. We believe that presenting our results in the spatial setting and making use of purely order-theoretic and topological techniques is useful in making the connections with this area more transparent  and in making these results more easily accessible to its community of researchers.

The paper is organized as follows: in Section 2 we give the basic definitions and properties of groupoids and quantales; in Section 3 we introduce our main groupoid setting of pairs $(\g, \s)$ and their associated {\em groupoid quantales} $\Q(\g, \s)$; in Section 4 we introduce the  {\em SGF-quantales}: these quantales are to  groupoid quantales what locales are to topologies. In the same section, a procedure is defined to associate a set groupoid $\g(\Q)$ with every SGF-quantale $\Q$. This procedure is based on the {\em incidence relation} and its properties, which are detailed in Subsection 4.1. In Section 5 we introduce the {\em spatial} SGF-quantales, prove that this class includes the groupoid quantales $\Q(\g, \s)$, and that if $\Q$ is spatial, then the set of units of $\g(\Q)$ can be made into a sober space. In Section 6 we prove that the back and forth correspondence between  spatial SGF-quantales and  the pairs $(\g, \s)$ is bijective. In Section 7 we explain in detail why, although it is not so by definition, this correspondence is compatible with the \'etale setting of \cite{Re07}. In Section 8 we conclude with two concrete examples.

\section{Preliminaries}

\subsection{Strongly Gelfand quantales}

\label{sec:preliminaries} A {\em quantale} $\Q$
\cite{Mu86, Ro90} is a complete join-semilattice endowed with
an associative binary operation $\cdot$ that is completely
distributive in each coordinate, i.e.

D1: $c\cdot \bigvee I = \bigvee\{c\cdot q : q\in I\} $

D2: $\bigvee I \cdot c = \bigvee\{q\cdot c : q\in I\}$

\noindent for every $c\in \Q$, $I\subseteq \Q$. Since it is a
complete join-semilattice, $\Q$ is also a complete, hence bounded,
lattice. Let $0, 1$ be the lattice bottom and top of $\Q$,
respectively. Conditions $D1$ and $D2$ readily imply that $\cdot$ is
order-preserving in both coordinates and, as $\bigvee\varnothing =
0$,
 that $c\cdot 0 = 0 = 0\cdot c$ for every $c\in \Q$. $\Q$ is {\em unital} if
there exists an element $e\in \Q$ for which

U: $e\cdot c = c = c\cdot e$ for every $c\in \Q$,

\noindent and is {\em involutive} if it is endowed with a unary
operation $\ast$ such that, for every $c, q\in \Q$ and every
$I\subseteq \Q$,

I1: $c^{\ast\ast} = c$.

I2: $(c\cdot q)^{\ast} = q^{\ast}\cdot c^{\ast}$.

I3: $(\bigvee I)^{\ast} = \bigvee\{q^{\ast} : q\in I\}.$

%\noindent A {\em homomorphism} of (involutive) quantales is a map
%$h: \Q\rightarrow \Q'$ that preserves $\bigvee$, $\cdot$ (and
%$\ast$). If $\Q, \Q'$ are unital, then $h$ is {\em (strictly)
%unital} if $e'\leq h(e)$ ($h(e) = e'$).
%
%If $h(1_{\Q})= 1_{\Q'}$ then $h$ is {\em strong}.
\vskip1mm\noindent Relevant examples of unital
involutive quantales are:\\
{\em 1.} The quantale $\p(R)$ of subrelations of a given equivalence
relation $R\subseteq X\times X$.
\\
{\em 2.} The quantale $\p(G)$, for every group $G$.
\\
{\em 3.} Any frame $\Q$, setting $\cdot:=\wedge$, $\ast:=\ $id and
$e:= 1_{\Q}$.\\
 A {\em homomorphism} of (involutive) quantales is a map
$\varphi: \Q\rightarrow \Q'$ that preserves $\bigvee$, $\cdot$ (and
$\ast$). If $\Q, \Q'$ are unital quantales, then $\varphi$ is {\em
unital} if $e'\leq \varphi(e)$ and is {\em strictly unital} if $\varphi(e) = e'$. Notice that since every
homomorphism is completely join-preserving, then $\varphi(0) = \varphi(\bigvee
\varnothing) = \bigvee \varnothing = 0$. However, a homomorphism of
quantales does not need to preserve the lattice top. For example, if
$R\subset S$ are equivalence relations on  $X$, then the inclusion
$\p(R)\to \p(S)$ is a strictly unital homomorphism of quantales
that does not. If $\varphi(1_{\Q})= 1_{\Q'}$ then $\varphi$ is {\em strong}.

\noindent Let $\Q$ be a unital involutive quantale. An element
$f\in \Q$ is {\em functional} if $f^\ast\cdot f\leq e$ and is a {\em
partial unit}  if both $f$ and $f^\ast$ are
functional\footnote{If $\Q = \p(R)$ for some equivalence relation
$R\subseteq X\times X$, then functional elements (partial units) are
exactly the graphs of
 (invertible) partial maps $f$ on $X$. %and  a binary relation $R$
%is the graph of a function if $(x, y), (x, z)\in R$ imply $y=z$.
%Then $R$ is not the graph of a function iff there exist $y, z\in
%dom(R)$ such that $y\neq z$ and $(y, z)\in R^{-1}\circ R$, iff
%$R^{-1}\circ R\not\subseteq\Delta$.
}. The set of functional elements (resp.\ partial units) will be denoted by
$\f(\Q)$ (resp. $\ii(\Q)$). It is easy to verify that $e\in \ii(\Q)$ and
$\ii(\Q)$ is closed under composition and involution of $\Q$.
%i.e. it is  an {\em involutive submonoid} of $\Q$.
 Moreover, if $f\leq g\in
\ii(\Q)$ then $f\in \ii(\Q)$.
%Functional invertible elements are meant to capture the algebraic behaviour of local bisections in groupoids.

 Let $\Q_e =\{c\in \Q : c\leq e\}$.
$\Q_e\subseteq \ii(\Q)$, moreover, $\Q_e$ is a unital involutive
subquantale of $\Q$.
\begin{definition}
A unital involutive quantale $\Q$ is {\em strongly Gelfand} (or an {\em SG-quantale}) if
\begin{itemize}
\item[SG.] $a\leq a\cdot a^\ast\cdot a$ for every $a\in\Q$.
\end{itemize}
\end{definition}
\noindent Recall that $\Q$ is  a {\em Gelfand} quantale (see also \cite{Ro90})
if $a=a\cdot a^\ast\cdot a$ for every right-sided element of $\Q$
($a\in \Q$ being {\em right-sided} if $a=a\cdot 1$). It is immediate
to see that  every SG-quantale is Gelfand, and that $f=f\cdot
f^\ast\cdot f$ for every SG-quantale $\Q$ and every $f\in \f(\Q)$.
%The following properties hold in any SG-quantale:
We will simplify  notation and write $a\cdot b$ as $ab$.\\
A quantale $\Q$ is  {\em supported} if it is endowed with a {\em support}, which is a completely join-preserving map $\varsigma:\Q\to \Q_e$ s.t.\
$\varsigma (a)\leq aa^\ast$ and $a\leq \varsigma(a) a$ for every $a\in \Q$.  For every supported quantale $\Q$, $\Q_e$ coincides with $\varsigma\Q$ and it is a locale with $ab = a\wedge b$ and trivial involution (cf.\ \cite[Lemma II.3.3]{Re06}). It is immediate to see that every supported quantale is an SG-quantale. Therefore the item 1 of the following proposition shows that the fundamental property of supported quantales  mentioned above generalizes to SG-quantales.
Even more importantly, the items 3 and 4 of the following proposition show that  the crucial connection between supported quantales and inverse monoids \cite[Theorem II.3.17.1]{Re06} generalizes to SG-quantales\footnote{We thank Pedro Resende for pointing to our attention this interpretation of items 1 and 3 of Proposition \ref{prop:qe}.}:
\begin{prop}\label{prop:qe} For every SG-quantale $\Q$,
\begin{enumerate} \item the subquantale $\Q_e$ is a frame: in particular,
involution $\ast$ coincides with the identity, and composition
$\cdot$ with $\wedge$.
\item For every $f,g\in\f(\Q)$ such that $f\leq g$,  $f=g\ $ iff $\ f f^\ast=g
g^\ast$.
\item $\ii(\Q)$ is an inverse monoid\footnote{An \emph{inverse
semigroup} (cf. \cite{Pat99}) is a semigroup such that for every
element $x$ there exists a unique inverse, i.e.\ an element $y$ such
that $x = xyx$ and $y = yxy$. Equivalently, an inverse semigroup is
a semigroup such that every element has some inverse and any two
idempotent elements commute. An {\em inverse monoid} is an inverse semigroup with a multiplicative unit.} whose set of idempotents coincides with $\Q_e$, and whose natural order coincides with the order inherited from $\Q$.
\item The assignment $\Q\mapsto \ii(\Q)$ extends to a functor $\ii$ from the category of SG-quantales to the category of inverse monoids.
\end{enumerate}
\end{prop}

\begin{proof}
1. Let $d\leq e$. By SG, $d\leq d d^\ast d\leq e d^\ast e =
d^\ast$, and likewise, $d^\ast\leq d$, hence involution is identity
on $\Q_e$. If $c\leq e$, then $cc = c$: indeed, $cc\leq ce = c$, and
by SG and the fact that involution is identity on $\Q_e$, $c =
cc^\ast c = (cc)c\leq (cc) e = cc$. Let $d_1, d_2\leq e$. Then
$d_1d_2\leq d_1 e = d_1$ and $d_1d_2\leq e d_2 = d_2$, so
$d_1d_2\leq d_1\wedge d_2$. Conversely, if $c\leq d_1$ and $c\leq
d_2$, then $c = cc \leq d_1d_2$, hence $d_1\wedge d_2\leq d_1d_2$.
\\
2. By SG and since $f\leq g$ implies $f^\ast\leq g^\ast$, $g =
gg^\ast g\leq ff^\ast g\leq f g^\ast g\leq f e = f$.
\\
3. By SG, $ff^\ast f= f$ and $f^\ast ff^\ast = f^\ast$ for every
$f\in \ii(\Q)$. Hence, it is enough to show that the restriction of
the product to the idempotent elements of  $\ii(\Q)$ is commutative.
This follows from item 1 above and from the fact that for every
$f\in \ii(\Q)$, $ff = f$ iff $f\leq e$: Indeed, if $f\leq e$, then
by (1), $ff = f\wedge f = f$. Conversely, if $ff= f$, then $f^\ast =
(ff)^\ast = f^\ast f^\ast$, hence $ff^\ast = ff^\ast f^\ast \leq
ef^\ast =f^\ast$, and so $f = ff^\ast f \leq f^\ast f\leq e$. Since $\Q_e\subseteq \ii(\Q)$,
this also shows that the set of idempotent elements of $\ii(\Q)$ coincides with $\Q_e$. Hence,
the natural order of the inverse monoid $\ii(\Q)$ is defined as follows: $f\leq g$ iff $f= gh$ for some $h\in \Q_e$,
and therefore it coincides with the order inherited from $\Q$.\\
4. Every strict homomorphism of unital involutive quantales maps partial units to partial units,
hence it restricts to a homomorphism of inverse monoids.
\end{proof}
\commment{
\noindent If $\Q= \p(R)$ for some equivalence relation $R$, then
$\Q$ is SG %The unit $\Delta$ is the functional invertible element
%arising as the graph of the identity map, and so every element in
%$\Q_e$
and for every $f\in \f(\Q)$, $ff^\ast$ can be identified with the
domain of the partial map of which $f$ is the graph. So item 2 of
the previous lemma says that if $f$ and $g$ are partial maps and $g$
is a restriction of $f$, then $f$ and $g$ coincide iff their domains
coincide.
}

\subsubsection{A natural action}
\label{subsec:action}
For every SG-quantale $\Q$,  a natural action\footnote{In \cite{Re07} (discussion before Lemma 4.5) a similar action is defined on the whole of a stable quantale $\Q$ on $\Q_e = \varsigma \Q$ by the assignment $(a, h)\mapsto \varsigma (ah)$, which makes $\Q$ into a $\varsigma\Q$-module.} can be defined of the inverse semigroup $\ii(\Q)$ on
$\Q_e$: indeed, for every
 $f\in \ii(\Q)$ and every $h\in \Q_e$ let $h^f=f^\ast h f$.
This is indeed an action of  $\ii(\Q)$ because of the identity
$(h^f)^g=h^{fg}$.
\begin{lemma}\label{lemma:action}
For every $h\in\Q_e$ and $f\in\ii(\Q)$,
\begin{enumerate}
\item $hf = fh^f$ and $f^\ast h = h^f f^\ast$.
\item If $h\leq f f^\ast$ then $h=f h^f f^\ast$.
\end{enumerate}
\end{lemma}
\begin{proof}
1. Since $f=ff^\ast f$ and because the product is commutative
in $\Q_e$, we get $h f=h \cdot (ff^{\ast}) f = (ff^{\ast})\cdot h f
= f h^f$. The second equality goes analogously 2. Immediate.
\end{proof}

\subsection{Groupoids}

\begin{definition}
\label{def:groupoid discrete} A {\em  set groupoid} is a tuple
$\mathcal{G}=(G_0, G_1,m,d,r,u,i)$, s.t.\ :

\noindent
G1. $G_0$ and $G_1$ are sets;

\noindent
G2. $d,r:G_1\to G_0$ and $u:G_0\to G_1$ s.t. $d(u(p)) = p = r(u(p))$ for
every $p\in G_0$;

\noindent
G3. $m:(x,y)\mapsto xy$ is an associative map defined on $G_1\times_0 G_1 = \{(x,y)\ |\ r(x)=d(y)\}$ and s.t. $d(xy) = d(x)$ and $r(xy) = r(y)$;

\noindent
G4. $xu(r(x)))=x = u(d(x))x$ for every $x\in G_1$.;

\noindent
G5.the map $i:G_1\to G_1$ denoted as $i(x) = x^{-1}$ is s.t.
$xx^{-1}=u(d(x))$,  $x^{-1}x=u(r(x))$, $d(x^{-1})= r(x)$ and
$r(x^{-1})= d(x)$ for every $x\in G_1$. \\
The identities in G2-G5 can be equivalently summarized by saying that the following diagram commutes:

\begin{center}
\begin{tikzpicture}
\draw[black, ->]  (0.3, 0.2) .. controls (1, 0.6) and (2, 0.6) .. (2.8, 0.2);%
\draw[black, ->]  (0.3, -0.2) .. controls (1, -0.6) and (2, -0.6) .. (2.8, -0.2);%
\draw[black, ->]  (-0.13, 0.16) .. controls (-0.3, 0.55) and (0.3, 0.55) .. (0.13, 0.16);%

\draw[black, <-]  (0.4, 0) -- (2.6, 0);%
\draw[black, <-]  (-0.4, 0) -- (-2, 0);%
\draw (0, 0.65) node {\small{$i$}};
\draw (1.5, 0.65) node {\small{$d$}};
\draw (1.5, -0.38) node {\small{$r$}};
\draw (1.5, 0.12) node {\small{$u$}};
\draw (-1.3, 0.12) node {\small{$m$}};

\draw (-3, 0) node {$G_1\times_0 G_1$};
\draw (0, 0) node {$G_1$};
\draw (3, 0) node {$G_0$};
\end{tikzpicture}
\end{center}

\end{definition}
\begin{exa}
\label{ex: groupoids}
$\quad$ \\
\noindent
{\em 1.} For any equivalence relation $R\subseteq X\times X$, the tuple $(X, R, \circ, \pi_1, \pi_2, \Delta,
()^{-1})$ defines a groupoid. Of particular interest are versions of this examples where $X$ is a topological space: for instance, the space of Penrose tilings \cite{Co94,MuRe05,PaRe10} is such an example and its associated groupoid is \'etale.

\noindent
{\em 2.} For any group $(G,\cdot, e, ()^{-1})$, the tuple
$(\{e\}, G, \cdot, d, r, u, ()^{-1})$ is a groupoid,
and the equalities G4 and G5
just restate the group axioms.

\noindent
{\em 3.} The following example is a
special but important case of the first one: every
topological space $X$ can be seen as a groupoid by setting
$G_1=G_0=X$ and identity structure maps. In this case,
$G_1\times_{0}G_1=\{(x,x)\ |\ x\in X\}$ and $xx =x$ for every
$x\in X$.

\noindent
{\em 4.} A groupoid can be associated with any action\footnote{For any group $G$, a {\em (left) action} of $G$ on a set $X$ is a function $\cdot: G\times X\rightarrow X$
s.t.\ for all $g, h \in G$ and $x \in X$, $(gh)·x = g·(h·x)$ and
$e·x = x$  ($e$ being the identity of $G$).} $G\times X\to X$ of a group $G$ on a set $X$, by setting
$G_1=G\times X$, $G_0=X$, and for all $g, h\in G$ and $x, y\in X$,  $d(g,x)=x$,  $r(g,x)=gx$,  $u(x)=(e,x)$  ($e\in G$ being the identity element), and $(g,x)\cdot(h,y)=(hg,x)$ whenever $y=gx$.

\noindent
{\em 5.} To a group action as above, another groupoid can be associated, which is given by the equivalence relation $R\subset X\times X$ defined by $xRy$ iff there exists some $g\in G$ such that $y=gx$.
\end{exa}

\noindent Some useful facts about groupoids are reported in the following:

\begin{lemma}
\label{lemma:properties of groupoids}
 For all $p\in G_0$, $x, y\in G_1$,\\ %\begin{enumerate}
% \item
1. $u(p)^{-1} = u(p)$,\\
%\item
2. $x = xx^{-1}x$ and $x^{-1} = x^{-1}xx^{-1}$,\\
%\item
3. if $xy^{-1}, x^{-1}y\in u[G_0]$ then $x=y$,\\
%\item
4. if $x = xyx$ and $yxy=y$, then $y = x^{-1}$,\\
%\item
5. $(x^{-1})^{-1}=x$,\\
%\item
6. $(xy)^{-1}= y^{-1}x^{-1}$.
%\end{enumerate}
\end{lemma}
\commment{
\begin{proof}
1. By  G4, G5 and G2, $u(p)^{-1} = u(p)^{-1}u(r(u(p)^{-1})) = u(p)^{-1}u(d(u(p))) = u(p)^{-1}u(p) = u(r(u(p)))= u(p)$.

\noindent
2. By G5 and G4, $xx^{-1}x = u(d(x))x = x$. The second one is
analogous.
%and $x^{-1} =
%u(d(x^{-1}))x^{-1}= u(r(x))x^{-1} = x^{-1}x x^{-1}$.

\noindent
3. Assume $xy^{-1}= u(p)$ and $x^{-1}y= u(q)$ for some $p, q\in
G_0$. Then, by G2, G3 and G5, $p= d(u(p))= d(xy^{-1}) = d(x)$ and $p
= r(u(p))= r(xy^{-1}) = r(y^{-1}) = d(y)$. Likewise, using the
second part of the assumption, one shows that $r(x)=q = r(y)$.
Hence, by G5, $x^{-1}x =u(r(x))= u(q)= x^{-1}y = u(r(y))= y^{-1}y$
and $xx^{-1}=u(d(x))= u(p)= xy^{-1} = u(d(y))= yy^{-1}$. By G4 and
item 2 above, $x = u(d(x))x u(r(x))= u(d(y)) xx^{-1} y = u(d(y)) y
y^{-1}y = u(d(y))y = y$.

\noindent
4. Since $xyx$ is well defined by assumption, $r(x)=d(y)$ and $d(x)=
r(y)$, so $x^{-1}x = x^{-1} xyx = u(r(x))yx = u(d(y))yx = yx = y
xu(r(x)) = y x u(d(y))= yx yy^{-1} = yy^{-1}$, and likewise $xx^{-1}
= xy = y^{-1}y$, therefore $y= yy^{-1}y=x^{-1}xy= x^{-1}xx^{-1} =
x^{-1}$.

\noindent
5. Immediate consequence of items 2 and 4 above.

\noindent
6. By (4), it is enough to show that
$(y^{-1}x^{-1})(xy)(y^{-1}x^{-1}) = y^{-1}x^{-1}$ and
$(xy)(y^{-1}x^{-1})(xy) = xy$. As $r(x)= d(y)= r(y^{-1})$,
$y^{-1}x^{-1}x = y^{-1}u(r(x))= y^{-1}u(r(y^{-1})) = y^{-1}$, so by
(2), $(y^{-1}x^{-1})(xy)(y^{-1}x^{-1}) = y^{-1}yy^{-1}x^{-1}=
y^{-1}x^{-1}$.
\end{proof}
}

\noindent For every groupoid $\g$, $\p(G_1)$ can be given
the structure of a unital involutive quantale  (see also
\cite{Re06} and \cite{Re07} 1.1 for a more detailed discussion): indeed, the product
and involution on $G_1$ can be lifted to $\p(G_1)$ as follows:
\begin{center}
$S\cdot T=\{x\cdot y\ |\ x\in S,\ y\in{T}$ and $r(x)=d(y)\}\quad $
 $\quad S^\ast=\{ x^{-1}\ |\ x\in S\}$. \end{center}
\noindent Denoting by $E$ the image of the structure map $u:G_0\to
G_1$, we get:
\begin{fact}\label{fact:discrete group quantale}
$\langle\p(G_1), \bigcup, \cdot, ()^{\ast}, E\rangle$ is a strongly Gelfand quantale.
\end{fact}
\begin{proof}
SG follows from Lemma \ref{lemma:properties of groupoids}.2.
\end{proof}

\section{SP-groupoids and their quantales}
In what follows, a \emph{groupoid} is a set groupoid $\g=(G_0,G_1)$ s.t.\ $G_0$ is a sober  space\footnote{For every topological space $X$, a closed set $C$ is {\em irreducible} iff $C\neq\varnothing$ and for all closed sets $K_1, K_2$, $C\subseteq K_1\cup K_2$ implies that $C\subseteq K_1$ or $C\subseteq K_2$. A {\em sober} space is a topological space s.t.\ its irreducible closed sets are exactly the topological closures of singletons.}. For every $p\in G_0$, let $\overline{p}$ denote the topological closure of $\{p\}$. The topology on $G_0$ will be denoted by $\Omega(G_0)$.  We do not fix any a priori topology on $G_1$.

\begin{definition}\label{def:localbisection} A {\em local bisection} of a groupoid $\g$ is a map $s:U\to G_1$ such that $d\circ s=\mbox{id}_U$ and $t=r\circ s$ is a partial homeomorphism $t:U\to V$ between open sets of $G_0$. A {\em bisection image}\footnote{Images of local bisections are sometimes referred to as $G$-{\em sets} (cf. \cite{Renault}). However, since ``$G$-sets'' usually refer to sets equipped with a group action, we propose an alternative name here.} of $\g$ is the image of some local bisection of $\g$.
Let $\s(\g)$ be the collection of the bisection images  of $\g$.\end{definition}
\noindent Notice that since $d\circ s=\mbox{id}_U$, local bisections are completely determined by their corresponding bisection images. We will denote bisection images by $S, T$, possibly indexed, and their corresponding local bisections will be $s, t$, possibly indexed.

\vspace{2pt}

\noindent Since $G_1$ is not endowed with any topology, the local bisections according to the definition above are not required to be continuous, as is the case e.g.\ in \cite{PrRes09,Re07}. This design choice can be motivated as follows. First, there exists at least a topology on $G_1$ w.r.t.\ which the local bisections of Definition \ref{def:localbisection} are always continuous, and it is defined as follows:  Let $R\subseteq G_0\times G_0$ be the  equivalence relation  induced by $G_1$ and let $\pi: G_1\to R$ be the map defined as $\pi(x)=(d(x),r(x))$;  the open subsets of $G_1$ are those of the form  $\pi^{-1}[A]$, for any open subset $A\subseteq R$ in the product topology inherited from  $G_0\times G_0$. This topology is in general not even $T_0$. However, even if $\s$ is defined as the family of local bisections that are continuous w.r.t.\ some given topologies on $G_0$ and on $G_1$, if the resulting topological groupoid $\g$ is not \'etale, then the quantale $\Q(\g, \s)$ in Definition \ref{def:TGQ} below will contain the topology of $G_1$ as a subquantale but will not coincide with it, nor will this topology be  uniquely identifiable inside $\Q(\g, \s)$. So the topology on $G_1$ is a piece of information that cannot be retained along the back-and-forth correspondence defined in this paper. On the other hand, the absence of topology on $G_1$ allows for a greater generality: for instance, $G_1$ can be taken  as a set endowed with a measure (typically a Haar measure) and, correspondingly, the local bisections can be taken as measurable maps defined on open sets of $G_0$. This could be interesting in view of possible applications of this setting to the theory of $C^\ast$-algebras. Also, not assuming any topology on $G_1$ allows in principle for a greater choice of selection bases (cf.\ Example \ref{ex:selection bases}).

Finally, the groupoids as we understand them in this paper  can always be made into \'etale topological groupoids, by endowing $G_1$ with the topology generated by taking the intersections of bisection images as a subbase. However, their associated inverse quantal frames turn out to be in general much larger than the quantales we associate with these (non topological) groupoids. The comparison with \cite{Re07}, which we will discuss more in detail in Section 7, is based on a special case of this observation. 

The statements in the  following proposition are well known for other settings and readily follow from the definition of bisection image:
\begin{prop}
\label{prop: properties of gsets}
For every groupoid $\g$ and every bisection image $S$ of $\g$,
\begin{enumerate}
\item $\s(\g)\subseteq \ii(\p(G_1))$.
\item $(\s(\g), \cdot, ()^\ast, E)$ is an inverse monoid.
\end{enumerate}
\end{prop}
%\noindent It readily follows from the definition of local bisection that for every groupoid $\g$ and every $G$-set $S$, $S\cdot S^\ast\subseteq E$ and $S^\ast\cdot S\subseteq E$, i.e.\

\begin{exa}
\label{ex: local bisections}
{\em 1.} Let $(X,X\times G)$ be as in Example \ref{ex: groupoids}.4
%4 section \ref{sect:groupoids},
s.t.\ moreover $X$ is  a locally connected topological space and $G$ a group with the discrete topology. Then the local bisections are the locally constant maps $U\to G$ s.t.\ $U\subseteq X$ is an open set.\\ %Hence $\Q(\g)$ \texttt{not defined yet!!!} coincides with the product topology on $G_1=G\times X$ and it is obviously \'etale.\\
{\em 2.}
On the other hand, let $R\subseteq X\times X$ be the equivalence relation induced by the action of $G$ as in Example \ref{ex: groupoids}.5. If  $R$ is endowed with the quotient topology induced by the map $(d,r):G\times X\to R$, defined by $(g,x)\mapsto (x,gx)$, then the first projection $\pi_1:R\to X$ is not necessarily \'etale.
For example, let $X=\C$ and $G=\{z\in\C\ |\ z^n=1\}$ be the group of the $n$th roots of the unity, for $n\geq 2$. Consider the action of $G$ on $X$ given by the product $(z,x)\mapsto zx$. Its induced equivalence relation is $R=\{(x,y)\ |\ y=zx,\ z\in G\}$. The relation $R$ can be seen as a groupoid as in Example \ref{ex: groupoids}.1.  For every $z, w\in G$ s.t.\ $z\not=w$ consider the  local bisections of  $R$ defined respectively by $x\mapsto (x,zx)$ and $x\mapsto (x,wx)$. Their images intersect only at $(0,0)\in R$, so
$d:R\to X$ is not \'etale. % (which implies, by the results in \cite{Re07} and the discussion in Section 7, that the groupoid quantale $\Q(R)$ is not a frame).
\end{exa}

\begin{definition}\label{def:SPG} A groupoid $\g=(G_0,G_1)$ as above has the {\em selection property}, or is an {\em SP-groupoid}, %with $G_0$ a topological space
  if
$G_1$ is covered by bisection images. %i.e.\ for any $y\in G_1$ there exists a local bisection $s$ such that $y\in \mbox{Im}(s)$
\end{definition}

\commment{
It is well known that:
\begin{fact}
\label{fact: subquantale generated by monoid } If $\Q = (Q, \bigvee,
\cdot, ^{\ast}, e)$ is a unital involutive quantale and $S\subseteq
Q$ is such that $(S, \cdot, ()^\ast, e)$ is an inverse monoid,
%monoid\footnote{contains $\ e$ and is closed under $\cdot$ and $ ^{\ast}$},
%that $\langle S, \cdot,()^{\ast}, e\rangle$ is an inverse monoid,
then the sub-semilattice of $(Q, \bigvee)$ generated by $S$ is a
unital involutive subquantale of $\Q$.
\end{fact}
The fact above suggests that,}

\noindent Given a groupoid $\g$, we can  associate a unital involutive quantale with every inverse monoid
$\s\subseteq \p(G_1)$: namely,  the quantale defined as the sub join-semilattice of
$\langle\p(G_1), \bigcup\rangle$ generated by $\s$.  However, in our non \'etale setting, we may not be able to reconstruct back the inverse semigroup from the quantale. For this, we need  the following new, stronger definition:

\begin{definition}\label{def:selfam} A \emph{selection base} for an SP-groupoid $\g$ is a family $\mathcal{S}\subseteq \s(\g)$ verifying the following conditions:\\
SB1. $\mathcal{S}$ is a sub inverse monoid of $\s(\g)$;\\
SB2. $u[U]\in \mathcal{S}$ for every open set $U\in \Omega(G_0)$;\\
SB3. if $\{S_i\}_{i\in I}\subseteq \s$ and $S_i\cdot S_j^\ast\subseteq E$ and $S_i^\ast\cdot S_j\subseteq E$ for every $i, j\in I$, then $\quad \quad \bigcup_{i\in I}S_i\in \s$.\\
SB4. For every $S, T\in \s$, $\{p\in G_0\mid s(p) = t(p)\}$ is the union of locally closed\footnote{ For every topological space $X$, a subset $Y\subseteq X$ is  {\em locally closed } if $Y = U\cap C$ for some open set $U$ and some closed set $C$.} subsets of $G_0$.\\
SB5. $\mathcal{S}$ covers $G_1$.
%i.e.\ for any $y\in G_1$ there exists some $s\in S$ such that $y\in \mbox{Im}(s)$.
\end{definition}
\noindent Selection bases are not in general topological bases, cf.\ Subsection 8.2 for an example.
% una sel basis e' un bi fascio su G_0 il cui spazio totale e' G_1
%It is well known (e.g.\ \cite{Pat99}) that the collection $\s(\g)$ of $G$-sets is a unital inverse semigroup, hence if $\g$ is an SP-groupoid then
%\noindent %Clearly, $\s(\g)$ is the largest selection base of any SP-groupoid $\g$. \texttt{questo non e' piu' vero}
\begin{exa}
\label{ex:selection bases}
\noindent
1. A continuous group action $G\times X\to X$ as in Example \ref{ex: groupoids}.4 gives rise to a canonical selection base consisting of the bisection images corresponding to local bisections $s_g: U\to G_1$ defined by the assignment  $x\mapsto g\cdot x$ for any $g\in G$.

\noindent
2. If $G_0$ is a $T_1$ space, then the family $\s(\g)$ of the local bisections  is the greatest selection base.
Notice that in this case the condition SB5 is trivially verified, since any subset of $G_0$ is the union of its singleton subsets, which are all closed.

\noindent
3. Let $X$ be  a $T_1$-space  with a continuous group action as above, and let $R$ be the equivalence relation induced by the group action, as in Example \ref{ex: local bisections}.2. Then the groupoid $(X, R)$ has the following, in general distinct, selection bases: the one consisting of the local bisections of the form $x\mapsto(x,gx)$ for those $g\in G$ s.t.\ the assignment $x\mapsto gx$ defines a locally constant\footnote{A map $f: X\to Y$ is {\em locally constant} if for every $x\in X$ there exists some open neighborhood $U$ of $x$ s.t.\ $f(x) = f(y)$ for every $y\in U$.} map, and $\s(R)$, which is given by local bisections of the form $x\mapsto(x,g(x)x)$, such that $g(x)$ is not necessarily locally constant, but the assignment $x\mapsto g(x)x$ defines a partial homeomorphism.
\end{exa}

%Nell'esempio di actions di Penrose ci sono 3 selection base: una e' quella fatta dalle restrizioni ad aperti di local bisections del tipo x\cdot g al variare di x in U aperto di X (sono costanti su ): questo viene etale perche' la mappa di struttura d viene l'unione disgiunta di copie di X. in questo caso le local bisections sono tutte e sole  Un'altro e' la rel di eq indotta dall'azione: qui S e' la famiglia di local bisections x,g\cdot x dove g e' el costante del gruppo e x in U (queste sono le loc bisec associate agli el di G) pero' ce ne sono altre: una x\mapsto x g_x\cdot x pero' la mappa g_x\cdot x deve essere omeomorphismo locale (caso particolare: g_x costante) la terza: tutti i G-sets: Domanda: la prima e la seconda dovrebbero dare risultati diversi come X prendo la circonferenza S_1 come gruppo prendo S_1 = U_1 che sono le rotazioni gruppo

\begin{remark}
\label{rem:towards pointfree gener}
{\em The pairs $(\g, \s)$ can be regarded as categories on the topology of $G_0$, in the following way. Let $\s_\g$ be the category having the elements of $\Omega(G_0)$ as objects, and such that for every $U, V\in \Omega(G_0)$,  Hom$_{\s_\g}(U,V)$ is the set of those  $s\in\s$ (identified with their associated local bisections) such that $r[s[U]]\subseteq V$. This category includes the frame $\Omega(G_0)$ as a sub-category, and axiom SB3 says that the functor Hom$_{\s_\g}(-,U)$ is a sheaf on $\Omega(G_0)$.\\ This observation paves the way to a generalization of the present results to a setting of quantales associated with sheaves on locales, which will be developed in \cite{PaRe10a}.  Axiom SB4, which is needed in the present setting (see proof of Proposition \ref{prop:Q(G, S) is spatial}), will be always true in the localic setting. Indeed the subspace where two elements of $\s$ ``intersect" each other can still be defined, but any subspace of a locale is a join of locally closed subspaces, (cf.\ \cite{MM92}, chapter IX pp. 504, 505 for a discussion on the canonical subspace associated with a given local operator).} \end{remark}
%\begin{lemma}Let $\g$ be an SP-groupoid and $\s$ be one of its selection bases. If $S, T\in \s$ and $p\in G_0$ s.t.\ $s(p) = t(p)$, then there exists some $H\in \Omega(G_0)$ s.t.\ $s(q) = t(q)$ for every $q\in H\cap \overline{p}$.\end{lemma}
%\begin{proof}
%By SB4, $\{q\in G_0\mid s(q) = t(q)\} = U\cap C$ for some open set $U$ and some closed set $C$. Then take $H = U$.
%\end{proof}

\subsection{Groupoid quantales}

\begin{definition}\label{def:TGQ} For every SP-groupoid $\g$ and every selection base $\mathcal{S}$ for $\g$, the  \emph{groupoid quantale} ({\em GQ} for short) $\Q(\g,\mathcal{S})$ associated with the pair $(\g, \mathcal{S})$ is the sub $\bigcup$-semilattice of $\p(G_1)$ generated by  $\mathcal{S}$.\end{definition}

%\texttt{Question: do we get different quantales by varying the base? }

\noindent In particular, the elements of $\Q(\g,\mathcal{S})$ are arbitrary joins of elements of $\mathcal{S}$.

\noindent Condition SB3 crucially guarantees that  $\s$ can be traced back from $\Q(\g, \s)$:

 \begin{prop}
 \label{prop:S=I(Q(G, S))}
 For every SP-groupoid $\g$ and every selection base $\mathcal{S}$ for $\g$,  \\
1. $\s = \ii(\Q(\g,\mathcal{S}))$.\\
2. $\Q(\g, \s)_e = \{u[U]\mid U\in \Omega(G_0)\}$.\\
3.  $P\in \Q(\g, \s)_e$ is prime iff $P = u[G_0\setminus \overline{p}]$ for some $p\in G_0$.
 \end{prop}
 \begin{proof}
1.  $\s\subseteq \ii(\Q(\s, \g))$ immediately follows from $\s\subseteq \s(\g)\subseteq \ii(\p(G_1))$ (cf.\ Proposition \ref{prop: properties of gsets}.1). Let $T\in \ii(\Q(\s, \g))$, so $T = \bigcup\{S_i\}_{i\in I}$ for some collection $\{S_i\}_{i\in I}\subseteq \s$. Then for every $i, j\in I$,  $S_i\cdot S_j^\ast\subseteq T\cdot T^\ast\subseteq E$ and $S_i^\ast\cdot S_j\subseteq T^\ast\cdot T\subseteq E$, hence by SB3, $T = \bigcup_{i\in I}S_i\in \s$.\\
2. By SB2, if $U\in \Omega(G_0)$, then $u[U]\in \s$ and clearly $u[U]\subseteq u[G_0]$ so $u[U]\in \Q(\g, \s)_e$. Conversely, let $H\in \Q(\g, \s)_e$; then $H\subseteq u[G_0]$ and $H = \bigcup_{i\in I} S_i$ for some $\{S_i\mid i\in I\}\subseteq \s$. Let $s_i: U_i\to G_1$ be the corresponding local bisections. Then $S_i\subseteq u[G_0]$ implies that $s_i(p) = u(p)$ for every $p\in U_i$, therefore $H = u[U]$ for $U = \bigcup \{U_i\mid i\in I\}.$\\
3. The prime elements of $\Omega(G_0)$ are exactly the complements of irreducible closed sets, and by assumption $G_0$ is sober.
%For every closed set $C$ and every point $p\in G_0$, $p\in C$ iff $\overline{\{p\}}\subseteq C$. From this it readily follows that $u[G_0\setminus \overline{\{p\}}]$ is prime. Conversely, let $P$ be prime. Then $P\neq u[G_0]$ so $u(p)\notin P$ for some $p\in G_0$. It is easy to show that $P\subseteq u[G_0\setminus \overline{\{p\}}]$. Suppose that $u(q)\notin P$ for some $q\in G_0\setminus \overline{\{p\}}$. Then
 \end{proof}
 
 \begin{exa}
{\em 1.} Let $\g = (X,X\times G)$ be as in Example \ref{ex: local bisections}.1,
%4 section \ref{sect:groupoids},
%i.e.\ $X$ is  a locally connected topological space, $G$ a group with the discrete topology; since 
and let $\s$ be the selection base associated with all local bisections (the locally constant maps $U\to G$ s.t.\ $U\subseteq X$ is an open set). $\Q(\g, \s)$  coincides with the product topology on $G_1=G\times X$ and the resulting topological groupoid is \'etale.\\
{\em 2.}
On the other hand, let $R\subseteq X\times X$ be as in Example \ref{ex: local bisections}.2, and let $\s$ be the selection base associated with all local bisections. The observations in \ref{ex: local bisections}.2  imply, by the results in \cite{Re07} and the discussion in Section 7, that the groupoid quantale $\Q(R, \s)$ is not an inverse quantal frame.
\end{exa}

 \section{SGF quantales and their set groupoids}

% We are now ready to introduce the class of abstract quantales that we will represent as  topological groupoid quantales:
%\subsection{SGF quantales}

\begin{definition}\label{def:SGF}
An \emph{SGF quantale} is a unital involutive quantale $\Q$ satisfying the following extra axioms:

\noindent
{\em SGF1.} $\Q$ is $\bigvee$-generated by $\ii(\Q)$.

\noindent {\em SGF2.} $f = ff^\ast f$ for every $f\in \ii(\Q)$.

\noindent {\em SGF3.} For any $f,g\in\ii(\Q)$ and $h\in\Q_e$ if
$f\leq h\cdot 1\vee g$ then $f\leq h\cdot f\vee g$.
\end{definition}
\vskip2mm \noindent

\noindent Clearly, the first two axioms imply that every SGF-quantale is SG. Let us motivate the axioms by showing that every groupoid quantale is  SGF:
\begin{prop}
\label{prop:Q(G, S) is SGF}
For every SP-groupoid $\g$ and every selection base $\s$ of $\g$, $\Q(\g, \s)$ is an SGF-quantale.
\end{prop}
\begin{proof}
SGF1 readily follows from  $\s\subseteq\ii(\Q(\s, \g))$. SGF2 follows from Fact \ref{fact:discrete group quantale}.
For  SGF3, let $F, G\in \ii(\Q(\g, \s))$ and $H\in \Q(\g, \s)_e$. Proposition \ref{prop:S=I(Q(G, S))} implies that $F, G$ are bisection images: let them correspond respectively to the local bisections $f$ and $g$. From the same proposition it follows that $H$ can be identified, via the structure map $u$, with some open subset $h\in \Omega(G_0)$. Assume that $F\subseteq H\cdot 1 \cup G$. This implies that for every $x\in dom(f)$, either $x\in h$, hence $f(x)\in H\cdot F$, or $x\in dom(g)$, which implies, since $d\circ f =id = d\circ g$, that $f(x) = g(x)$. %\texttt{usare che $d\circ s = id$}
%and let us show that $F\subseteq H\cdot F\cup G$:  This implies that if some point $x\in G_0$ does not belong to the open set
%$h\subseteq G_0$ then $f(x)=g(x)$. But then for any point
%$x\in G_0$ either $x\in h$, and so $f(x)\in hf$, or $f(x)=g(x) $.
\end{proof}
\noindent The remainder of this section is aimed at constructing the set groupoid associated with any SGF-quantale.  This construction is based on the incidence relation, whose definition (Definition \ref{def:inters}) and properties are given in the following subsection.

%\subsection{From SGF quantales to groupoids}
%It was mentioned early on that  SGF-quantales model algebraically groupoid quantales just like locales do with topologies. In this section we will show that any SGF-quantale which satisfies an appropriate spatiality axiom is isomorphic to some GQ. This new axiom  is a direct generalization of the notion of spatiality for locales.
\subsection{The incidence relation on SGF-quantales}
Let $\Q$ be an SGF-quantale. For every $f\in\ii(\Q)$ let
$d(f)=ff^\ast$ and $r(f)=f^\ast f.$ The following lemma lists some straightforward but useful formal properties of these abbreviations:
\begin{lemma}
Let $\Q$ be an SGF-quantale,  $f, f', g\in\ii(\Q)$ and $h, k\in \Q_e$. Then: \\
%\begin{enumerate}
%\item
1.  $d(hf) = hd(f)$ and $r(fk) = r(f)k$.\\
%\item
2. If $f\leq g$, then $d(f)\leq d(g)$ and $r(f)\leq r(g)$.\\
%\item
3. $d(ff') = d(f')^{f^\ast}$ and $r(ff') = r(f)^{f'}$.\\
%\item
4. $r(f) = d(f)^f$ and $d(f) = r(f)^{f^\ast}$.
%\end{enumerate}
\end{lemma}
%Since the product is commutative and idempotent in $\Q_e$, it is easy to see that Moreover, if  and
\noindent Let $\p_e$ be the set of the {\em prime} elements of $\Q_e$ (cf.\ \cite{Jo82}) i.e.\ those non-top elements $p\in \Q_e$ s.t.\ for every $h, k\in \Q_e$, if $h\wedge k\leq p$, then $h\leq p$ or $k\leq p$.  %$\p_e=\{p\in\Q_e\ |\ p \mbox{ prime in }\Q_e\}
Let $$\ii=\{(p,f)\in \p_e\times \ii(\Q)\ |\ p\in\p_e,\ d(f)\not\leq p\}.$$
\noindent For every $p\in \p_e$, $p\leq e$ and $p\neq e$ imply that $d(e) = e\nleq p$, hence $(p, e)\in \ii$. Moreover, if $f\leq g$ then $d(f)\leq d(g)$ so $(p, f)\in \ii$ implies that $(p, g)\in \ii$.
For every $h\in \Q_e$, $\Q_h=\{k\in\Q\ |\ k\leq h\}$  is a subframe of $\Q_e$.

\begin{lemma}\label{lemma:transform} For every $f\in\ii(\Q)$,
\begin{enumerate} \item the assignment $h\mapsto h^f=f^\ast h f$ defines a frame isomorphism $()^f: \Q_{d(f)}\to\Q_{r(f)}$, the inverse of which is defined by $k\mapsto k^{f^\ast}=fkf^\ast$.
\item The prime elements of $\Q_{d(f)}$ correspond bijectively to the prime elements of $\Q_{r(f)}$ via $()^f$.
\end{enumerate}
\end{lemma}
\begin{definition}\label{def:inters} The \emph{incidence relation}
$\sim$ on $\ii$ is defined by setting
$$(p,f)\sim (q,g) \mbox{  iff } p=q \mbox{ and } h\not\leq p \mbox{ and } hf\leq pf\vee g\mbox{ for some } h\leq
d(f)\wedge d(g).$$
We will also alternatively write  $f\sim_p
g$ (read: $f$ and $g$ are incident in $p$) in place of $(p,f)\sim (q,g)$.\end{definition}
\noindent \emph{Remark.} Let us interpret the incidence relation if $\Q = \Q(\g, \s)$ for some SP-groupoid $\g$ and some selection base $\s$: in this case, by Proposition \ref{prop:S=I(Q(G, S))}, $\Q_e$ can be identified via $u$ with $\Omega(G_0)$, $\p_e$ can be identified with the collection $\{\overline{p}^c\mid p\in G_0\}$ of the complements of the closures $\overline{p}$ of points  $p\in G_0$ and $\ii(\Q) = \s$. For all $F, G\in \ii(\Q)$, let $f, g$ be their associated local bisections: then $F\sim_{\overline{p}^c} G$ iff there exists an open subset $H$ of $G_0$ s.t.\ $H\cap \overline{p}\neq \varnothing$ (i.e., since $p$ is dense in $\overline{p}$,  $p\in H$), s.t.\ $f$ and $g$ are both defined over $H$ and coincide over $H\cap \overline{p}$.
Moreover, if $G_0$ is $T_1$, then  $\p_e $ corresponds to the collection of the complements of points of $G_0$ and $F\sim_{\{p\}^c} G$ iff
%there exists some open subset $H\in \Omega(G_0)$ s.t.\ $H\subseteq dom(f)\cap dom(g)$, $x\in H$ and $f[H]\subseteq $
$f(p)=g(p)$.\\
Notice
also that the relation $f\sim_p g$ may be defined by saying that
there exist some $f'\leq f$ and $g'\leq g$ s.t.\ $d(f')=d(g')\not\leq p$
and $f'\leq pf'\vee g'$.
\begin{prop}\label{prop:equivalence} 1. The relation $\sim$ is an equivalence
relation.\\
2. If $f\sim_p g$ and $g\leq g' $ then $f\sim_p g'$.
\end{prop}
\begin{proof} 1. Reflexivity is obvious. Symmetry: $hf\leq pf\vee g$ implies $h  = h\wedge d(f) = hd(f)\leq pd(f)\vee gf^\ast$ hence $h\leq p\vee fg^\ast$ and so $hg\leq pg\vee fg^\ast g\leq pg\vee f$. Transitivity: If $h_1f\leq pf\vee g$ and $h_2g\leq pg\vee l$ then setting $h=h_2h_1$ we get $h\not\leq p$, because $p$ is prime, $h\leq d(f)\wedge d(l)$ and $$hf = h_2h_1f\leq h_2pf\vee h_2g\leq h_2pf\vee pg\vee l\leq p(f\vee g)\vee l\leq p\cdot 1\vee l.$$
Hence, by SGF3, $ hf\leq phf\vee l\leq pf\vee l$.\\
2. Straightforward. \end{proof}
\begin{lemma}
\label{lemma:p and f[p]}
For every SGF-quantale $\Q$, let $(p,f)\in\ii$. Then:\\
%\begin{enumerate}
%\item
1. there exists a unique $q\in \p_e$,
denoted $q=f[p]$, s.t.\ $r(f)\not\leq q$ and $pf=fq$.\\
%\item
2. For every $h\in \Q_e$, if $h\not\leq p$, then $h^f\nleq f[p]$.\\
%\item
3. For every $h\in \Q_e$, if $h\not\leq p$, then $d(hf)\not\leq p$ and $r(hf)\nleq f[p]$.\\
%\item
4. If $f\sim_p g$ then  $f[p]=g[p]$.\\
%\item
5. $(f[p], f^\ast)\in \ii$ and $f^\ast[f[p]] = p$.\\
%\item
6. If $(f[p], g)\in \ii$, then $(p, fg)\in \ii$ and $fg[p] = g[f[p]]$.\\
%\item
7. $ff^\ast\sim_p e$ and $f^\ast f\sim_{f[p]} e$.
%\end{enumerate}
\end{lemma}
\begin{proof} 1. From the basic theory of locales, we recall that for every $0\not=h\in\Q_e$
$p'$ is a prime element of $\Q_h$ iff
$p' = hp$ for a unique $p\in\p_e$ s.t.\ $h\nleq p$.
\begin{comment}: indeed, it is straightforward to show that any such $hp$ is  prime in $\Q_h$. Conversely, if $p'$ is prime in $\Q_h$, then take $p = \bigvee\{k\in \Q_e\mid hk\leq p'\}$: since $hp'  = p' \leq p'$, then $hp = \bigvee\{hk\mid k\in \Q_e \mbox{ and } hk\leq p'\} = p'$. Moreover, $h\nleq p$, for if $h\leq p$, then $p' = hp = h = \top_{\Q_h}$, contrary to the assumption of $p'$ being a prime element of $\Q_h$. Let us show that $p\in \p_e$: if $k_1k_2\leq p$ then $(hk_1)(hk_2)\leq hp = p'$, and since $p'$ is prime in $\Q_h$ let us assume w.l.o.g.\ that $hk_1\leq p'$. This implies that $k_1\in \{k\in \Q_e\mid hk\leq p'\}$ and so $k_1\leq \bigvee\{k\in \Q_e\mid hk\leq p'\} = p$.\\
For uniqueness, let us show that for every prime element $p'$ in $\Q_h$, if $p'=hp=hq$ for some $p,q\in\p_e$, then $p=q$: indeed since $hq\leq p$ and $h\nleq p$,  $p$ being prime implies that $q\leq p$, and symmetrically, $p\leq q$. \\
Let $(p, f)\in \ii$, and let $p':=pd(f)$. Then it is easy to see, since $f = ff^\ast f$, that $p^f = {p'}^f$; moreover, by the remark above, $p'$ is a prime in $\Q_{d(f)}$.
%, and
\end{comment}
By Lemma \ref{lemma:transform}.2, ${p'}^f$ is a prime of $\Q_{r(f)}$, hence ${p'}^f = r(f)q$ for a unique $q\in \p_e$ s.t.\ $r(f)\nleq q$. Therefore, by Lemma \ref{lemma:action}.1,  $pf= fp^f = f{p'}^f = fr(f)q=fq$.\\
2. Let $q = f[p]$. Then $pf = fq$, so $h^f\leq q$ implies, by Lemma \ref{lemma:action}.1, that $hf = fh^f\leq fq = pf$, hence $hd(f)\leq pd(f) \leq p$, and since $p$ is prime and $d(f)\nleq p$, then $h\leq p$.\\
3. Recall that $d(hf) = hd(f)$. Since $p$ is prime and $d(f)\nleq p$, then $h\not\leq p$ implies that $d(hf)=hd(f)\not\leq p$. Let $k=h^f$ (so, by Lemma \ref{lemma:action}.1, $hf=fk$); by item 2 we get that $k\nleq f[p]$, and since $f[p]$ is prime  and $r(f)\nleq f[p]$, then $r(hf)= r(fk) = r(f)k\not\leq f[p]$.\\
4. Assume that $f\sim_p g$. Then there exists some $h\in \Q_e$ s.t.\ $h\nleq p$, $h\leq d(f)\wedge d(g)$ and $hg\leq pg\vee f$. %Then clearly $d(g)\nleq p$, for otherwise $h\leq d(g)\leq p$; hence $(p, g)\in \ii$.
Let $q = f[p]$ and $q' = g[p]$, so $pf = fq$ and $pg = gq'$, and let us show that $q =q'$. Our assumption implies that there exists some $h\in \Q_e$ s.t.\  $h\nleq p$ and $hgq\leq pgq\vee fq=pgq\vee pf\leq p\cdot 1\vee pgq$, hence by SGF3, $hgq\leq p\cdot hgq\vee pgq\leq  pg=gq'$. This implies that  $r(hg)q = r(hgq)\leq r(gq') = r(g)q'\leq q'$. Since $q' = g[p]$ and $h\nleq p$, by item 3 we get that $r(hg)\nleq q'$. Hence, since $q'$ is prime, $r(hg)q \leq q'$ implies that $q\leq q'$. The proof that  $q'\leq q$ is obtained symmetrically, from $g\sim_p f$.  \\
5. By item 2, $d(f)\nleq p$ implies that $d(f^\ast) = r(f) = d(f)^f\nleq f[p]$, which proves that  $(f[p], f^\ast)\in \ii$. Let $q = f[p]$.  In order to show that $f^\ast[q] = p$, by the uniqueness of $f^\ast[q]$ it is enough to show that $qf^\ast = f^\ast p$, which readily follows from $pf = fq$.\\
6. Let $q = f[p]$. Then by item 5, $p = f^\ast[q]$, hence $d(g)\nleq q$ implies by item 2 that $d(fg) = d(g)^{f^\ast}\nleq p$, which proves that $(p, fg)\in \ii$. Let $q' = g[q]$, so $qg = gq'$; to finish the proof it is enough to show that $pfg = fgq'$: since $pf = fq$, then $pfg = fqg = fgq'$.\\
7. By assumption $d(f)\nleq p$, so take $h = d(f)$: clearly $h\leq d(ff^\ast)\wedge d(e)$ and  $hff^\ast \leq e = pff^\ast\vee e$. The second relation follows from item 5 and the first relation in this item.
\end{proof}
\begin{prop}
\label{prop:sim compos e invol}
 For every SGF-quantale $\Q$, let $f, f'g, g'\in\ii(\Q)$ and $p\in \p_e$.
\begin{enumerate}\item  If $f\sim_p g$ and $f'\sim_{f[p]} g'$ then $ff'\sim_p gg'$.
\item  $f\sim_p g$ iff $f^\ast\sim_{f[p]} g^\ast$.
\end{enumerate}
\end{prop}
\begin{proof} 1. Let $q = f[p] = g[p]$ and $q' = f'[q] = g'[q]$. Hence $pf = fq$, $pg = gq$, $qf' = f'q'$ and $qg' = g'q'$. By assumptions there exist some $h, h'\in \Q_e$ s.t.\ $h\nleq p$, $h'\nleq q$, $h\leq d(f)\wedge d(g)$, $h'\leq d(f')\wedge d(g')$, $hf\leq pf\vee g$ and $h'f'\leq qf'\vee g'$. By Lemma \ref{lemma:p and f[p]}.5, $p = f^\ast[q] = g^\ast[q]$, hence $h'\nleq q$ implies, by Lemma \ref{lemma:p and f[p]}.2, that ${h'}^{f^\ast}\nleq p$ and ${h'}^{g^\ast}\nleq p$, and so $h{h'}^{f^\ast}{h'}^{g^\ast}\nleq p$. Let $k = h{h'}^{f^\ast}{h'}^{g^\ast}$: to finish the proof it is enough to show that $k\leq d(ff')\wedge d(gg')$ and $kff'\leq pff'\vee gg'$. $h'\leq d(f')$ implies that $k\leq {h'}^{f^\ast}\leq d(f')^{f^\ast} = d(ff')$, and analogously $k\leq d(gg')$, from which the first inequality follows. For the second inequality, $$kff'\leq h{h'}^{f^\ast} ff' = hfh'f'\leq pfqf'\vee pfg'\vee gqf'\vee gg'.$$ Since $gqf' = pgf'$, then $kff'\leq p(fqf'\vee fg'\vee gf')\vee gg'\leq p\cdot 1\vee gg'$. By SGF3, we get that $kff'\leq pkff'\vee gg'\leq pff'\vee gg'$.\\
2. Since $p = f^\ast[f[p]]$ it is enough to show the left-to-right direction. So let $q = f[p]$. By Lemma \ref{lemma:p and f[p]}.5, $qf^\ast = f^\ast p$. By assumption, there exists some $h\in \Q_e$ s.t.\ $h\nleq p$, $h\leq d(f)\wedge d(g)$ and $hf\leq pf\vee g$. Let $k = h^f$: then, by Lemma \ref{lemma:p and f[p]}.2, $k\nleq q$; moreover, $h\leq d(f)$ implies that $h^f\leq d(f)^f = r(f) = d(f^\ast)$ and likewise $k\leq d(g^\ast)$. Finally, $k f^\ast  = f^\ast h\leq f^\ast p\vee g^\ast = q f^\ast \vee g^\ast$.
\end{proof}

 \subsection{The set groupoid of an SGF-quantale}
\begin{definition}\label{def:groupoid}
For every SGF-quantale $\Q$, its associated set groupoid $\g(\Q)$  is defined as follows:
$G_0=\p_e$ and $G_1=\ii/\sim$, moreover, denoting the elements of $G_1$
by  $[p,f]$, the structure maps of $\g(\Q)$ are given by the following assignments: $$d([p,f])=p,\quad r([p,f])=f[p],\quad u(p)=[p,e],$$ $$[p,f][q,g]=[p,fg]\quad \mbox{ only if }\quad q=f[p]$$
 $$[p,f]^{-1}=[f[p],f^\ast].$$
\end{definition}

\begin{lemma}
The structure maps above are indeed well defined.
\end{lemma}
\begin{proof}
If $(p, f)\sim (p'f')$ then $p = p'$, so $d$ is well defined. Moreover, by Lemma \ref{lemma:p and f[p]}.4, $f[p] = f'[p']$, so $r$ is well defined.
Also by Lemma \ref{lemma:p and f[p]}.4, it is straightforward to see that if $(p,f)\sim (p', f')$ and $(q, g)\sim (q', g')$ then %$p = p'$, $q = q'$, $f[p] = f[p']$ and $g[q] = g'[q']$; so
$[p, f][q, g]$ is defined iff $q = f[p]$ iff $q' = f[p']$ iff $[p', f'][q', g']$ is defined; Proposition \ref{prop:sim compos e invol}.1 exactly says that the product is well defined. Likewise, Proposition \ref{prop:sim compos e invol}.2 exactly says that the inverse is well defined.
\end{proof}

\begin{prop}\label{prop:groupoid} For every SGF-quantale $\Q$, $\g(\Q)$ is a set groupoid.
\end{prop}
\begin{proof}
G2: Recall that for every $p\in \p_e$, $(p, e)\in \ii$; then $e[p] = p$ hence $d(u(p)) = d([p, e]) = p = e[p] = r([p, e]) = r(u(p))$. \\
G3: The associativity of the product readily follows from the definitions using Lemma \ref{lemma:p and f[p]}.6.  If $q = f[p]$, then, by Lemma \ref{lemma:p and f[p]}.6, $r([p, f][q, g]) = r([p, fg]) = fg[p] = g[f[p]] = g[q] = r([q, g])$. \\
G4: immediate from the definitions.\\
G5: $d([p, f]^{-1}) = d([f[p], f^\ast]) = f[p] = r([p, f])$; by Lemma \ref{lemma:p and f[p]}.5, $r([p, f]^{-1}) = r([f[p], f^\ast]) = f^\ast[f[p]] = p = d([p, f])$. By Lemma \ref{lemma:p and f[p]}.7, $[p, f][p, f]^{-1} = [p, f][f[p], f^\ast] = [p, ff^\ast] = [p, e] = u(p) = u(d([p, f]))$. Likewise,  $p = f^\ast[f[p]]$ implies that the product $[p, f]^{-1}[p, f]$ is well defined and by \ref{lemma:p and f[p]}.7 $[p, f]^{-1}[p, f] = [f[p], f^\ast f] = [f[p], e] = u(r([p, f]))$.
\end{proof}

\section{Spatial SGF-quantales and their SP-groupoids}

The last ingredient needed in $\g(\Q)$ is a topology on $G_0$. For this, we need a condition on $\Q$ which guarantees $\Q_e$ to be a spatial frame. The notion of spatial SGF-quantales that we are going to introduce in this section generalizes spatial locales,
i.e.\ the locales that are meet-generated by their prime elements. %Accordingly, the following definition provides us with a notion of generalized prime elements:
\begin{definition}
\label{def:SPQ}  For every SGF quantale $\Q$ and every $[p,f]\in \ii/\sim$, let
$$\ii_{[p,f]}=\{g\in \ii(\Q)\mid d(g)\leq p \mbox{ or } (p, g)\not\sim (p, f)\}\quad \mbox{ and }\quad I_{[p, f]} = \bigvee \ii_{[p, f]}.$$
$\Q$ is  \emph{spatial} if:\\
SPQ1. for every $(p,f)\in \ii$, $f\nleq I_{[p,f]} $.\\
SPQ2.
 For every $a\in \Q$, $a =\bigwedge \{I_{[p,f]}\mid a\leq I_{[p, f]} \}.$
\end{definition}
\noindent It immediately follows from the definition that $p\in \ii_{[p, f]}$, hence $p\leq I_{[p, f]}$ for every $(p, f)\in \ii$.  It is also immediate to see that $(p, f)\sim (p', f')$ implies that $\ii_{[p, f]} = \ii_{[p', f']}$, and that if $g\nleq I_{[p, f]}$ then $g\sim_p f$.

\begin{lemma}
\label{lemma:I_[p, f]}
For every SGF-quantale $\Q$ s.t.\ SPQ1 holds and every $g\in \ii(\Q)$,  $g\leq I_{[p, f]}$ iff  $g\in \ii_{[p, f]}$.
\end{lemma}
\begin{proof}
The right-to-left direction is clear. Conversely, if $g\leq I_{[p, f]}$ and $g\sim_p f$, then $I_{[p, f]} = I_{[p, g]}$ so $g\leq I_{[p, g]}$, i.e.\ $g\wedge I_{[p, g]} = g$, contradicting SPQ1.
\end{proof}
\noindent An immediate consequence of this lemma is that if $g\sim_p f$ then $g\nleq I_{[p, f]}$ (so indeed these two conditions are equivalent).
Let us verify that the  axioms for spatial quantales  are sound:
\begin{prop}
\label{prop:Q(G, S) is spatial}
For every SP-groupoid $\g$ and every selection base $\s$ of $\g$, the SGF-quantale $\Q(\g, \s)$ is spatial.
\end{prop}
\begin{proof}
Recall that $\ii(\Q(\g, \s)) = \s$ and the prime elements of $\Q(\g, \s)_e$ are exactly those $P=u[G_0\setminus \overline{p}]$ for $p\in G_0$ (cf.\ Proposition \ref{prop:S=I(Q(G, S))}). For every $F\in \s$, let $f: U_f\to G_1$ be its corresponding local bisection; in particular for every $H\in \Q(\g, \s)_e$, its corresponding local bisection is the restriction of the structure map $u$ to some open subset of $G_0$ that we denote $H$ as well. Then $HF$ is the image of $f_{|H}$, wherever defined. Moreover, $\ii_{[P, F]}$ (resp.\ $I_{[P, F]}$) is (the union of) the collection of all the $G\in \s$ corresponding to local bisections $g: U_g\to G_1$ s.t.\ either $U_g\cap \overline{p} = \varnothing$ (i.e.\ $p\notin U_g$) or $HG\not\subseteq PG\cup F$ for every  open set $H$ s.t.\ $p\in H\subseteq U_f\cap U_g$. \\ %and $H\cap\overline{p}\neq \varnothing$.\\
SPQ1: Let $P=u[G_0\setminus \overline{p}]$; it is enough to show that $f(p)\not\in I_{[P,F]}$. Suppose that $f(p)\in I_{[P,F]}$; then there exists some $g$ such that $G\not\sim_{P}F$ and $g(p)=f(p)$. By SB4, $p\in H\cap C\subseteq\{q\in G_0\mid f(q) = g(q)\}$ for some $H$ open and $C$ closed subsets of $G_0$. Then $H\cap\overline{p}\subseteq H\cap C \subseteq \{q\in G_0\mid f(q) = g(q)\}$. This means that $g_{|H}$ coincides with $f$ outside of $P$. In other words, $HG\subseteq PG\cup F$, contradicting the hypothesis that $G\not\sim_{P} F$.\\
%If $\overline{p} = \{p\}$, then $P=G_0\setminus\{p\}$ and $G=PG\cup \{g(p)\}\subseteq PG\cup F$, contradicting the hypothesis that $G\not\sim_{P} F$. Assume that  $q\in \overline{p}$ for some point $q\neq p$.  To finish the proof we need to show that there exists some open set $H$ such that $H\cap \overline{p}=\{p\}$: indeed if so, then $HG\subseteq PG\cup F$, hence  $G\sim_{P} F$,  contradiction. Notice that if $q\in\overline{p}$ and $q\neq p$ then, because $G_0$ is sober, $\overline{q}$ is a proper subset of $\overline{p}$; hence $P$ is properly included in its corresponding prime open set $Q = u[G_0\setminus \overline{q}]$. Let  $H=\bigcup \{Q\mid Q \mbox{ is prime and }  P\subsetneqq Q\}$. Because by assumption the collection of primes of which $H$ is the union is nonempty, then $H\not\subseteq P$, i.e.\ $H\cap\overline{p}\neq \varnothing$, i.e.\ $p\in H$. Moreover, if $q\in\overline{p}\setminus\{p\}$, its associated prime element $Q$  is included in $H$ by construction, hence $q\not\in H$. This shows that $H\cap \overline{p}=\{p\}$.\\
SPQ2: Let $A\in \Q(\g, \s)$ and let $G\in \s$ s.t.\  $G\not\subseteq A$. Then $g(p)\not\in A$ for some $p\in U_g$. Let $P = u[G_0\setminus\overline{p}]$, and let us show that if $F\in \s$ and $F\subseteq A$, then $F\in \ii_{[P, G]}$: indeed, if $F\subseteq A$ and $p\in U_f$ then $f(p)\neq \g(p)$, since by assumption $g(p)\not\in A$, therefore $F\not\sim_{P} G$. By SGF1, this shows that $A\subseteq I_{[P, G]}$. Since by Lemma \ref{lemma:I_[p, f]}, $G\not\subseteq\ii_{[P, G]}$, then $G \not\subseteq \bigwedge\{I_{[Q, G']}\mid A\leq I_{[Q, G']}\}$, which concludes the proof of the non trivial inclusion.
\end{proof}
%\begin{remark}\label{remark:incidence} SBAGLIATO, VALE SOLO PER T_1
%From the proof of axiom SPQ1 given above one deduces also the following interpretation of the incidence relation: $f\sim_p g$ if and only if $f(p)=g(p)$. Indeed in the proof above one finds $p$ some open set $H\subset G_0$ containing $p$ such that $H\cap \overline{p}=\{p\}$. If $h\in\Q_e$ is the element corresponding to $H$, then one has $f\sim_p g$ iff $hf\sim_p g$, which is equivalent to $f(p)=g(p)$. \end{remark}

\begin{lemma}
\label{lemma:g are special meets}
 If $\Q$ is spatial then for every $g\in \ii(\Q)$, $$g =\bigwedge \{I_{[p,f]}\mid d(g)\leq p \mbox{ or } (p, g)\not\sim (p, f) \}.$$
%\begin{equation}\label{eq:spatial} f=\bigwedge_{g\not\sim_p f}I_{[p,g]}.\end{equation}
\end{lemma}
\begin{proof}
Let  $g\in \ii(\Q)$. By {SPQ2} and Lemma \ref{lemma:I_[p, f]}, $g =\bigwedge \{I_{[p,f]}\mid g\leq I_{[p, f]}\} = \bigwedge \{I_{[p,f]}\mid g \in \ii_{[p, f]}\} = \bigwedge \{I_{[p,f]}\mid d(g)\leq p \mbox{ or } (p, g)\not\sim (p, f) \}.$
\end{proof}

%For every $g\in \ii(\Q)$, let $X_g = \{I_{[p,f]}\mid d(g)\leq p \mbox{ or } (p, g)\not\sim (p, f) \}$. Then it is clear that $g\in I_{[p, f]}$ iff $I_{[p, f]}\in X_g$.

\begin{prop}
\label{prop:Q_e spatial} If $\Q$ is spatial then $\Q_e$ is a spatial frame.\end{prop}
\begin{proof} Let $h\in \Q_e$ and let us show that $h = \bigwedge \{p\in \p_e\mid h\leq p\}$. Since $\Q$ is spatial, then by Lemma \ref{lemma:g are special meets}
$h=\bigwedge\{I_{[q,g]}\mid h\leq q \mbox{ or } (q, g)\not\sim (q, h) \}.$ \\
{\em Claim.} If $h\nleq q$, then  for every $g\in \ii(\Q)$ s.t.\ $(q, g)\in \ii$, if  $(q, g)\not\sim (q, h)$ then $(q, g)\not\sim(q, e)$.\\
From the claim it follows that $e\leq I_{[q, g]}$ for every $(q, g)\in \ii$ s.t.\ $h\nleq q$ and $(q, g)\not\sim (q, h)$.
Hence, $$h = h\wedge e=\bigwedge\{I_{[q,g]}\wedge e\mid h\leq q \mbox{ or } (q, g)\not\sim (q, h) \} = \bigwedge\{I_{[q,g]}\wedge e\mid h\leq q\}. $$
Since for every $(q, g)\in \ii$ s.t.\ $h\leq q$ there exists some $p = q$ s.t.\ $h\leq p$ and $p\leq I_{[q, g]}$, we can conclude that $$h \leq \bigwedge \{p\in \p_e\mid h\leq p\}\leq \bigwedge\{I_{[q,g]}\wedge e\mid h\leq q\} = h.$$
To finish the proof, we need to prove the claim: if $h\nleq q$ and $g\sim_q e$ then there exists some $k\in \Q_e$ s.t.\ $k\nleq q$, $k\leq d(g)\wedge e = d(g)$ and $kg\leq qg\vee e$. Let $h' = hk$: then $h'\nleq q$, $h'\leq d(g)\wedge h$ and $h'g\leq hqg\vee h\leq qg\vee h$; this shows that $g\sim_q h$.
\end{proof}

\begin{prop}
\label{prop:G(Q) is SP groupoid}
For every spatial SGF-quantale $\Q$,  \\
%\begin{enumerate}
%\item
1. every element $f\in\ii(\Q)$ corresponds to a local bisection of $\g(\Q)$.\\
%\item
2. $\g(\Q)$ is an SP-groupoid.
%\end{enumerate}
\end{prop}
\begin{proof}
1. By Prop \ref{prop:Q_e spatial}, every $h\in \Q_e$ can be identified with the open set $U_h = \{p'\in \p_e\mid h\nleq p'\}$ (cf.\ \cite{Jo82}). Then for every $f\in\ii(\Q)$, the map $s_f:U_{d(f)}\to \ii/\sim$ defined by  $s_f(p') = [p', f]$ is a local bisection of $\g(\Q)$: indeed, $d\circ s_f = id$ and it readily follows from Lemma \ref{lemma:p and f[p]}.2 and \ref{lemma:p and f[p]}.5 that $r\circ s_f$ is open and its inverse is $r\circ s_{f^\ast}$ which is also open.\\
2. If $[p, f]\in \ii$, then $[p, f]$ belongs to the bisection image corresponding to the local bisection $s_f$ defined above.
\end{proof}
\section{Spatial SGF-quantales are GQs}
\subsection{The canonical map}
\begin{definition}\label{def:alpha} For every SGF-quantale $\Q$, $\alpha:\Q\to\p(\ii/\sim)$ is defined by $$\alpha(a)=\{[p,f]\ |\ a\not\leq I_{[p,f]}\}.$$
\end{definition}
\begin{theorem}\label{thm:alpha} For every SGF-quantale $\Q$,

\noindent 1. $\alpha(\bigvee_i a_i)=\bigcup_i\alpha(a_i)$ for any
family $\{a_i|i\in I\}$ of elements of $\Q$.

\noindent 2. if $\Q$ is  spatial, then $\alpha$ is
an embedding.

\noindent 3. $\alpha(ab)=\alpha(a)\alpha(b)$ for any $a,b\in\Q$.

\noindent 4. $\alpha(a^\ast)=\alpha(a)^\ast$ for any $a\in\Q$.

\noindent 5. $\alpha(1)=G_1$ and $\alpha(e)=u(G_0)$.

\noindent
So $\alpha$ is a strong and strictly unital morphism of unital involutive quantales.
\end{theorem}
\begin{proof}
1. $[p, f]\in \bigcup_{i\in I}\alpha(a_i)$  iff $a_i\nleq I_{[p, f]}$ for some $i\in I$ iff $\bigvee_{i\in I} a_i\nleq I_{[p, f]}$ iff $[p, f]\in \alpha(\bigvee_{i\in I} a_i)$.\\
2. If $a\nleq b$, by SGF1 and SPQ2, $f\nleq I_{[q, g]}$ for some $f\in \ii(\Q)$ s.t.\ $f\leq a$ and some $(q, g)\in \ii$ s.t.\ $b\leq I_{[q, g]}$. Then $a\nleq I_{[q, g]}$, i.e.\ $[q, g]\in \alpha(a)$, and $[q, g]\notin \alpha(b).$\\
3. If $a\not\leq I_{[p,f]}$ and $b\not\leq I_{[q,g]}$ then by SGF1, $g_1\not\leq I_{[p,f]}$ and $g_2\not\leq I_{[q,g]}$ for some $g_1, g_2\in \ii(\Q)$ s.t.\ $g_1\leq a$ and $g_2\leq b$. Hence $g_1\sim_p f$ and $g_2\sim_q g$ and so if $q = f[p]$ then by Proposition \ref{prop:sim compos e invol}.1, $g_1g_2\sim_p fg$ which implies by Lemma \ref{lemma:I_[p, f]} that $g_1g_2\nleq I_{[p, fg]}$, i.e.\ $[p, fg]\in \alpha(g_1g_2)\subseteq \alpha(ab)$.
Conversely, if $ab\notin I_{[p, f]}$, then by SGF1 $g_1g_2\nleq I_{[p, f]}$ for some $g_1, g_2\in \ii(\Q)$ s.t.\ $g_1\leq a$ and $g_2\leq b$. Hence $g_1g_2\sim_p f$. In particular, $d(g_1g_2)\nleq p$, which implies, since $d(g_1g_2)\leq d(g_1)$, that $d(g_1)\nleq p$. So by Lemma \ref{lemma:p and f[p]}.1, let $q = f[p]$: then $pg_1 = g_1q$. Let us show that $d(g_2)\nleq q$: if not, then $qd(g_2) = d(g_2)$ and so $pg_1g_2 = g_1qg_2 = g_1qd(g_2)g_2 = g_1d(g_2)g_2 = g_1g_2$, hence $d(g_1g_2)\leq p$. From $d(g_1)\nleq p$ and $d(g_2)\nleq q$ we get $[p, g_1]\in \alpha(g_1)\subseteq\alpha(a)$ and $[q, g_2]\in \alpha(b)$.\\
4. If $a^\ast\not\leq I_{[p,f]}$ then by SGF1, $g^\ast\not\leq I_{[p,f]}$ for some $g\in \ii(\Q)$ s.t.\ $g\leq a$. Hence $g^\ast\sim_p f$, and so, by Proposition \ref{prop:sim compos e invol}.2, $g\sim_{f[p]} f^\ast$ which implies by Lemma \ref{lemma:I_[p, f]} that $g\nleq I_{[f[p], f^\ast]}$, i.e.\ $[p, f]^{-1} \in \alpha(g)\subseteq \alpha(a)$. Hence,  $[p, f] = ([p, f]^{-1})^{-1}\in \alpha(a)^\ast$. Conversely, if $[p, f]\in \alpha(a)^\ast$ then $a\nleq I_{[f[p], f^\ast]}$, then by SGF1 $g\nleq I_{[f[p], f^\ast]}$ for some $g\in \ii(\Q)$ s.t.\ $g\leq a$. Hence $g^\ast\leq a^\ast$ and $g\sim_{f[p]} f^\ast$, and so, by Proposition \ref{prop:sim compos e invol}.2, $g^\ast\sim_p f$ which implies by Lemma \ref{lemma:I_[p, f]} that $g^\ast\nleq I_{[p,f]}$, i.e.\ $[p, f]\in \alpha(g^\ast)\subseteq\alpha(a^\ast)$.
%The converse inclusion follows likewise from Proposition \ref{prop:sim compos e invol}.2 and Lemma \ref{lemma:I_[p, f]}.
\\
5. Since $f\nleq I_{[p, f]}$ for every $(p, f)\in \ii$, then $\alpha(1) = \{[p, f]\mid \bigvee \ii(\Q)\nleq I_{[p, f]}\} = G_1.$ For the second equality, $u(G_0)\subseteq \alpha(e)$ follows from $e\nleq I_{[p, e]}$ for every $p$. The converse inclusion follows from the fact that $e\nleq I_{[p, f]}$ by definition implies that $[p, f] = [p, e]$.
\end{proof}

\begin{prop}
\label{prop:alpha[I(Q)] selection basis}
For every spatial SGF-quantale $\Q$, $\alpha[\ii(\Q)]$ is a selection base of $\g(\Q)$.
\end{prop}
\begin{proof}
We already showed (see proof of Proposition \ref{prop:G(Q) is SP groupoid}) that for every $f\in \ii(\Q)$, $s_f: U_{d(f)}\to \p(\ii/\sim)$ defined by $s_f(p') = [p', f]$ is a local bisection of $\g(\Q)$. Let us show that $s_f[U_{d(f)}] = \alpha(f)$: $[p, g]\in s_f[U_{d(f)}]$ iff $[p, g] = s_f(p') = [p', f]$ for some $p'\in U_{d(f)}$ iff $g\sim_p f$ iff $f\nleq I_{[p, g]}$ iff $[p, g]\in \alpha(f)$. This shows that $\alpha[\ii(\Q)]$ is a collection of bisection images of $\g(\Q)$. \\
SB2: In particular, for every $h\in \Q_e$, $s_h[U_{d(h)}] = \alpha(h)$. Moreover, notice that $h\nleq p$ implies that $[p, h] = [p, e]$, from which it easily follows that $s_h[U_{d(h)}] = u[U_h]$.\\
SB1: it readily follows from Theorem \ref{thm:alpha} and Proposition \ref{prop:qe}.3.\\
SB3: it readily follows from Theorem \ref{thm:alpha}.3 and .4, and the fact that if $\{f_i\mid i\in I\}\subseteq \ii(\Q)$ s.t.\ $f_if_j^\ast\leq e$ and $f_i^\ast f_j\leq e$ then $\bigcup_{i\in I} f_i\in \ii(\Q)$. \\
SB4: Let $f, g\in \ii(\Q)$ and let $p\in \p_e$ be s.t.\ $s_f(p) = s_g(p)$. So $[p, f] = [p, g]$, i.e.\ $hg\leq pg\vee f$ for some $h\in \Q_e$ s.t.\ $h\nleq p$. Then, for every $q\in \p_e$ s.t.\ $p\leq q$ and $h\nleq q$, we get that $hg\leq pg\vee f\leq qg\vee f$, i.e.\ $s_f(q) = [q, f]=[q, g] = s_g(q)$. Since $h$ can be identified with an open set of $G_0$ and $p$ can be identified with the topological closure $\overline{p}$ of its corresponding point of $G_0$ (which we also denote by $p$), this is enough to show that $\{p\in G_0\mid s_f(p) = s_g(p)\}$ is the union of locally closed subsets of $G_0$. \\
SB5: Lemma \ref{lemma:I_[p, f]} readily implies that for every $(p, f)\in \ii$, $[p, f]\in \alpha(f)$.
\end{proof}
\subsection{The correspondence}
\begin{prop}
For every spatial SGF-quantale, $$\Q(\g(\Q), \alpha[\ii(\Q)]) \cong \Q.$$ For every SP-groupoid $\g$ and every selection base $\s$,
$$\g(\Q(\g, \s))\cong \g.$$
\end{prop}

\begin{proof}
Let $\Q$ be a spatial SGF-quantale. Then by Proposition \ref{prop:G(Q) is SP groupoid}, $\g(\Q)$ is an SP-groupoid, and by Proposition \ref{prop:alpha[I(Q)] selection basis}, $\alpha[\ii(\Q)]$ is a selection base of $\g(\Q)$. Then by definition $\Q(\g(\Q), \alpha[\ii(\Q)])$ is the sub $\bigcup$-semilattice of $\p(G_1) = \p(\ii/\sim)$ generated by $\alpha[\ii(\Q)]$. Theorem \ref{def:alpha} guarantees that  $\Q$ is isomorphic as a unital involutive quantale to its $\alpha$-image $\alpha[\Q]$ and hence that $\alpha[\Q]$ is $\bigcup$-generated by  $\alpha[\ii(\Q)]$. Hence by definition $\Q(\g(\Q), \alpha[\ii(\Q)]) = \alpha[\Q]$.\\
Let $\g$ be an SP-groupoid and $\s$ be a selection base for $\g$. Then by Propositions  \ref{prop:Q(G, S) is SGF} and \ref{prop:Q(G, S) is spatial} $\Q(\g, \s)$ is a spatial SGF-quantale. Moreover, by Proposition \ref{prop:S=I(Q(G, S))}, $\ii(\Q(\g, \s)) = \s$, $\Q(\g, \s))_e$ can be identified with the topology $\Omega(G_0)$, and since $G_0$ is sober, the prime elements of $\Q(\g, \s))_e$ bijectively correspond to the points of $G_0$ via the assignment $p\mapsto u[G_0\setminus \overline{p}]$. Since the prime elements of $\Q(\g, \s))_e$ form the space of units of $\g(\Q(\g, \s))$, then the assignment defines the map $\varphi_0$. Since $\s$ is a selection base, for every $x\in G_1$ $x = g(p)$ for some local bisection $g:U_g\to G_1$ s.t.\ its corresponding $G\in \s$ and some $p\in U_g$. Hence $(P, G)\in \ii$. Clearly, $[P, G] = [P', G']$ for any local bisection $g':U_g\to G_1$ s.t.\ $x = g'(p')$ for some $p'\in U_{g'}$, so the assignment $x\mapsto [P, G]$ defines a map $\varphi_1:G_1\to \ii/\sim$. The map $\varphi_1$ is bijective: indeed, if $(P, G)\in \ii$, then $p\in U_g$, so $[P, G] = \varphi_1(g(p))$; moreover, if $\varphi_1(x) = [P, F] = \varphi_1(y)$, then $x = f(p) = y$.
The fact that $(\varphi_0, \varphi_1)$ is indeed a morphism of groupoids is a standard if tedious verification.
\end{proof}

\section{Comparison with the \'etale localic setting}
The aim of this  section is showing informally that our bijective correspondence extends, in the spatial setting,  the   non functorial duality defined in  \cite{Re07} between localic \'etale groupoids and inverse quantal frames.
In \cite{Re07}  inverse quantal frames are defined as unital involutive quantales $\Q$ which are also frames for the lattice operations, are generated by $\ii(\Q)$ and have a {\em support}, i.e.\ a completely join-preserving map $\varsigma:\Q\to \Q_e$ s.t.\
$\varsigma (a)\leq aa^\ast$ and $a\leq \varsigma(a) a$ for every $a\in \Q$\footnote{It readily follows %(cf. \cite{Re06}, Lemma II.3.16)
that  for every $f\in\ii(\Q)$, $\varsigma(f)=f^\ast f$, hence $f = ff^\ast f$.}. %Considering such a quantale as a locale, which is possible because of the frame structure, Resende
Any such quantale is shown to be isomorphic to one of the form $\mathcal{O}(\g)$, for some localic \'etale groupoid\footnote{A {\em localic groupoid} is a groupoid in the category of locales. Such a groupoid is \'etale if $d$ is a partial homeomorphism.} $\g=(G_1,G_0)$. In particular, for any such $\g$, its associated quantale is based on the frame  $\mathcal{O}(\g)$, on which  the noncommutative product is defined by using the product  of $\g$ in the natural way.  When $\g$ is spatial (i.e.\ isomorphic to a topological groupoid), the back-and-forth correspondence in \cite{Re07} can be equivalently described in the following way. Recall that a $G$-set of a topological groupoid  is a subset $S\subseteq G_1$ such that the maps $d:S\to G_0$ and $r:S\to G_0$ are both homeomorphisms onto open subsets of $G_0$. A $G$-set\footnote{For the sake of highlighting the difference between the bisection images as defined in  Definition \ref{def:localbisection} and those of topological groupoids, in this section we will refer to the latter ones as $G$-sets.}  $S$ is therefore the image of a {\em continuous} local bisection $s:U\to G_1$, for some  open set $U$ of $G_0$. %\texttt{ qui mettere la differenza tra le nostre e le loro local bisections un po' piu' in evidenza}
 Then the inverse quantal frame associated with any \'etale topological groupoid $(G_1,G_0)$ can be equivalently described as the sub $\bigcup$-semilattice of $\p(G_1)$ generated by the $G$-sets of $G_1$. Conversely, if $\Q$ is an inverse  quantal frame corresponding to some spatial \'etale groupoid $(G_1,G_0)$, then $G_0$ can be equivalently recovered  as the topological space dual to the locale $\Q_e$, and  $G_1$  as the set of {\em germs} of elements of $\ii(\Q)$, i.e.\ as the set of  the equivalence classes of the relation $\sim$ on $\ii(\Q)$ defined as $f\sim g$ if and only if $hf=hg$ on some neighborhood $h$ of a point $p\in G_0$.

\vskip1mm
To show that the spatial version of the correspondence in  \cite{Re07} is a special case of our construction, we make the following remarks.
As we remarked early on, the notion of local bisection introduced in Definition \ref{def:localbisection} does not refer to any topology on $G_1$. However, if for some selection base $\mathcal{S}$  (Definition \ref{def:selfam}), the quantale $\Q(\g,\mathcal{S})$ as in Definition \ref{def:TGQ} happens to be an inverse quantal frame, then this quantale defines a topology on $G_1$. % which makes $(G_1,G_0)$ into a topological groupoid. To see this,
To continue the discussion, the following lemma will be useful:
 \begin{lemma}\label{lm:etale}
If $\Q$ is an inverse quantal frame, then for all $f, g\in \ii(\Q)$ and every $p\in \p_e$,  if $f\sim_p g$ then  there exists some $k\leq d(f)d(g)$ such that $k\not\leq p$ and $kf=kg$.
\end{lemma}
\begin{proof} By assumption, there exists some $h\in \Q_e$ s.t.\ $h\not\leq p$,  $h\leq d(f)d(g)$ and $hf\leq pf\vee g$.  Since $\Q$ is distributive, $hf=phf\vee (hf\wedge g)$. Let $kf=hf\wedge g$. Since $h=ph\vee k\not\leq p$, we get also $k\not\leq p$.  \end{proof}
\noindent Since  $\ii(\Q(\g,\mathcal{S})) = \s$ (cf.\ Proposition \ref{prop:S=I(Q(G, S))}), the Lemma above implies that $\s$ is a base for the topology $\Q(\g,\mathcal{S})$. %the graphs of two local bisections corresponding to elements in $\s$ intersect in the graph of a local bisection corresponding to an element of $\s$.
Then, notice that the elements of $\mathcal{S}$  are  images of local bisections that are {\em continuous} with respect to the given topology of $G_0$ and the topology $\Q(\g,\mathcal{S})$ on $G_1$: indeed, let $F\in \s$ and let $f$ be its associated local bisection. To show that $f$ is continuous it is enough to check that for every basic open $G\in \s$, $f^{-1}[G]\in \Omega(G_0)$, i.e.\ that every $p\in f^{-1}[G]$ has an open neighborhood that is contained in $f^{-1}[G]$. But this is again guaranteed by the Lemma. %is an open the intersection of two $G$-sets $f$ and $g$ in $\mathcal{S}$ is again a $G$-set in $\mathcal{S}$, that we will refer to as $hf$, for some open set $h$ of $u(G_0)$; if $s_f$ is the local bisection associated with $f$, then $s_f^{-1}(hf)=h$, which shows the continuity of $s_f$.
Therefore, bisection images defined as in Definition \ref{def:localbisection} are $G$-sets according to the standard definition. %defined by means of continuous local bisections with respect to the  groupoid $(G_1,G_0)$. %and $\mathcal{S}$ becomes the full set of $G$-sets in the topological context.
Analogously, it can be shown that the structure maps of the groupoid $(G_0, G_1)$ are continuous, i.e.\  $(G_0, G_1)$ is a topological groupoid.
By well known results in groupoid theory (cf.\ \cite{Renault}, chapter I, Definition 2.6,  Lemma 2.7 and
Proposition 2.8) this topological groupoid is \'etale.

Conversely, if $\Q$ is an inverse quantal frame, then it not difficult to see that $\Q$ is also an SGF-quantale.
If $\Q$ is also a spatial quantale as in Definition \ref{def:SPQ}, then its associated groupoid quantale $\g(\Q)$ (cf.\ Definition \ref{def:groupoid}) is defined by taking $G_1$ as the set of equivalence classes $[p,f]$ with respect to the incidence relation as in Definition \ref{def:inters}. %By Proposition \ref{prop:G(Q) is SP groupoid}.1, the functional invertible elements  $f,g\in\ii(\Q)$ correspond to local bisections of $\g=(G_1,G_0)$ in the sense of Definition \ref{def:localbisection}. %and they are incident at $p\in G_0$ if and only if they coincide over a locally closed neighborhood of $p$. %$f(p)=g(p)$, see Remark \ref{remark:incidence}.

Lemma \ref{lm:etale} shows that the equivalence classes $[p,f]$ coincide with the {\em germs} of local bisections, as in Definition \ref{def:localbisection}, (at $p$),
Also in this case, $\Q$ can be identified with a topology on $G_1$, via the canonical embedding $\alpha$ (cf.\ Theorem \ref{thm:alpha}), and by Proposition \ref{prop:alpha[I(Q)] selection basis}, $\alpha[\ii(\Q)]$ is a selection base; then axiom SB3 readily implies that $\alpha[\ii(\Q)]$ is collection of all $G$-sets, according to the standard definition, i.e.\ every $S\in \alpha[\ii(\Q)]$ is associated with a {\em continuous} local bisection. Hence the  construction of $G_1$ from $\Q$ in \cite{Re07} and in this paper coincide.  As before, from the same results in \cite{Renault}, the groupoid $(G_1, G_0)$ is \'etale.

\section{Examples}
In this section, we present two examples of groupoids arising as the equivalence relations induced by  group actions on topological spaces, as described in Example \ref{ex: local bisections}.2.
\subsection{Finite, not $T_1$ and \'etale}

Consider the finite topological space $X = (G_0, \Omega(G_0))$, defined as follows:
\begin{center} $G_0 = \{p_0, p_1, p_2\}$, $\Omega(G_0) = \{P_0 = \varnothing, H = \{p_0\}, P_1 = \{p_0, p_2\}, P_2 = \{p_0, p_1\}, G_0\}$.
\end{center}
So the opens are the down-sets of the  partial order on the left, and the lattice of the topology is represented on the right:
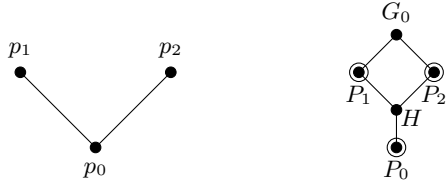
\begin{center}
\begin{tikzpicture}
\draw[black, -]  (0, 0) -- (1, 1);%
\draw[black, -]  (0, 0) --(-1, 1);%
\filldraw[black] (0,0) circle (2pt);
\filldraw[black] (-1,1) circle (2pt);
\filldraw[black] (1,1) circle (2pt);
\draw (0, -0.3) node {\small{$p_0$}};
\draw (-1, 1.3) node {\small{$p_1$}};
\draw (1, 1.3) node {\small{$p_2$}};

\filldraw[black] (4,1.5) circle (2pt);
\filldraw[black] (4.5,1) circle (2pt);
\draw[black] (4.5,1) circle (3.5pt);
\filldraw[black] (3.5,1) circle (2pt);
\draw[black] (3.5,1) circle (3.5pt);
\filldraw[black] (4,0.5) circle (2pt);
\filldraw[black] (4,0) circle (2pt);
\draw[black] (4,0) circle (3.5pt);

\draw[black, -]  (4, 0) -- (4, 0.5)-- (3.5, 1) -- (4, 1.5) --(4.5, 1)-- (4, 0.5);
\draw (4, -0.3) node {\small{$P_0$}};
\draw (3.5, 0.7) node {\small{$P_1$}};
\draw (4.5, 0.7) node {\small{$P_2$}};
\draw (4.2, 0.4) node {\small{$H$}};
\draw (4, 1.8) node {\small{$G_0$}};
\end{tikzpicture}
\end{center}
\noindent $X$ is clearly not $T_1$. The prime elements of $\Omega(G_0)$ are exactly $P_0$, $P_1$ and $P_2$, hence $X$ is sober.
The group acting on $X$ is $G = \{\varphi, id_{X}\}$, where $(\varphi(p_0) = p_0, \varphi(p_1) = p_2, \varphi(p_2) = p_1)$.
 The equivalence relation induced by the action of $G$ is then $$R = \{(p_0, p_0), (p_1, p_1), (p_2, p_2), (p_1, p_2), (p_2, p_1)\}.$$ The collection of partial homeomorphisms $X\to X$ consists of the restrictions to the open sets in $\Omega(G_0)$ of the  maps $\varphi$ and $id_{X}$. For every $H'\in \Omega(G_0)$,  $H'\varphi$ will denote the graph of the restriction of $\varphi$ to $H'$. The collection of the graphs of partial homeomorphisms $X\to X$ is
\begin{center}
$\s = \{H' = id_{H'}\mid H'\in \Omega(G_0)\}\cup\{H\varphi, P_1\varphi, P_2\varphi, \varphi\}$.
\end{center}
$X$ can be represented as the groupoid $\g = (X, R)$; then $\s$ is the collection of the bisection images of $\g$ and $\g$ is SP. Then $\Q(\g, \s)$ is the sub $\bigcup$-semilattice of $\p(R)$ generated by $\s$. Notice that for any two partial homeomorphisms of $X$ the set over which they coincide is an open set of $G_0$; this implies that the intersection of the graphs of any two partial homeomorphisms is again a graph of a partial homeomorphism, hence $\s$ is the base of a topology on $G_1$. So (cf.\ \cite{Pat99}, \cite{Re07}, and the discussion in Section 7)   $\g$ is \'etale  and $\Q(\g, \s)$ is an inverse quantal frame.

\subsection{Finite, not $T_1$ and not \'etale}

Consider the finite topological space $X = (G_0, \Omega(G_0))$, defined as follows:
\begin{center} $G_0 = \{p_0, p_1, p_2\}$, $\Omega(G_0) = \{\varnothing, P_1 = \{p_2\}, P_2 = \{p_1\},  P_0 = \{p_1, p_2\}, G_0\}$,
\end{center}
So the opens are the down-sets of the  partial order on the left, and the lattice of the topology is represented on the right:

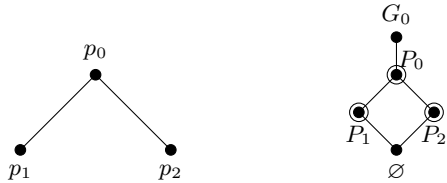
\begin{center}
\begin{tikzpicture}
\draw[black, -]  (-1, 0) -- (0, 1);%
\draw[black, -]  (1, 0) --(0, 1);%
\filldraw[black] (0,1) circle (2pt);
\filldraw[black] (-1,0) circle (2pt);
\filldraw[black] (1,0) circle (2pt);
\draw (0, 1.3) node {\small{$p_0$}};
\draw (-1, -0.3) node {\small{$p_1$}};
\draw (1, -0.3) node {\small{$p_2$}};

\filldraw[black] (4,1.5) circle (2pt);
\filldraw[black] (4.5,0.5) circle (2pt);
\draw[black] (4.5,0.5) circle (3.5pt);
\filldraw[black] (3.5,0.5) circle (2pt);
\draw[black] (3.5,0.5) circle (3.5pt);
\filldraw[black] (4,1) circle (2pt);
\filldraw[black] (4,0) circle (2pt);
\draw[black] (4,1) circle (3.5pt);

\draw[black, -]  (4, 1.5) -- (4, 1)-- (4.5, 0.5) -- (4, 0) --(3.5, 0.5)-- (4, 1);
\draw (4, -0.3) node {\small{$\varnothing$}};
\draw (3.5, 0.2) node {\small{$P_1$}};
\draw (4.5, 0.2) node {\small{$P_2$}};
\draw (4.2, 1.2) node {\small{$P_0$}};
\draw (4, 1.8) node {\small{$G_0$}};

\end{tikzpicture}
\end{center}
$X$ is clearly not $T_1$. The prime elements of $\Omega(G_0)$ are exactly $P_0$, $P_1$ and $P_2$, hence $X$ is sober.
The group acting on $X$ is $G = \{\varphi, id_{X}\}$, where $(\varphi(p_0) = p_0, \varphi(p_1) = p_2, \varphi(p_2) = p_1)$, and the equivalence relation induced by the action of $G$ is  $$R = \{(p_0, p_0), (p_1, p_1), (p_2, p_2), (p_1, p_2), (p_2, p_1)\}.$$
The collection of partial homeomorphisms $X\to X$ consists of the restrictions to the open sets in $\Omega(G_0)$ of the  maps $\varphi$ and $id_{X}$. For every $H\in \Omega(G_0)$,  $H\varphi$ will denote the graph of the restriction of $\varphi$ to $H$. The collection of the graphs of partial homeomorphisms $X\to X$ is
\begin{center}
$\s = \{H = id_{H}\mid H\in \Omega(G_0)\}\cup\{P_0\varphi, P_1\varphi, P_2\varphi, \varphi\}$.
\end{center}
$X$ can be represented as the groupoid $\g = (X, R)$; then $\s$ is the collection of the bisection images of $\g$ and $\g$ is SP. Notice that the set over which the graphs of $\varphi$ and $id_X$ coincide is $\{(p_0, p_0)\}$, which cannot be (nor contain) the graph of any (nonempty) partial homeomorphism since $\{p_0\}$  is closed but not open. Hence, $\s$ is not a topological base. %\texttt{cosi' si prova che S non e' chiuso rispetto a intersezioni finite, ma piu' di questo, non soddisfa le proprieta' di una base topologica}
Therefore,  (cf.\ \cite{Pat99}, \cite{Re07} and the discussion in Section 7)  $\g$ is not \'etale and  $\Q(\g, \s)$ is not a distributive lattice.

\noindent Alessandra Palmigiano \\
Institute for Logic, Language and Computation,\\
FNWI University of Amsterdam,\\
P.O. Box 94242\\
1090 GE Amsterdam\\
$<$a.palmigiano@uva.nl$>$

 \vspace{0.5cm}

 \noindent Riccardo Re \\
 Dipartimento di Matematica e Informatica,\\
 Universit\`a di Catania\\
 Viale Andrea Doria 6,\\
 95125 Catania, Italy.\\
 $<$riccardo@dmi.unict.it$>$

\end{document}